\documentclass[12pt,reqno]{amsart}
\usepackage[utf8]{inputenc}

\usepackage{graphicx}
\usepackage{xcolor}

\usepackage{calrsfs}

\usepackage[all]{xy}
\xyoption{2cell}
\UseAllTwocells

\usepackage[colorinlistoftodos]{todonotes}

\usepackage{tikz}
\usetikzlibrary{positioning,intersections,cd,matrix, arrows, backgrounds}
\usetikzlibrary{shapes,shapes.geometric,shapes.misc}
\usetikzlibrary{decorations}
\usetikzlibrary{decorations.pathreplacing}

\newcommand{\xyR}[1]{%
\xydef@\xymatrixrowsep@{#1}}

\newcommand{\xyC}[1]{%
\xydef@\xymatrixcolsep@{#1}}

\numberwithin{equation}{section}

\def\al{\alpha}
\def\be{\beta}

\def\la{\lambda}
\def\ro{\rho}
\def\si{\sigma}

\def\ta{\tau}
\def\up{\upsilon}

\def\ph{\phi}

\def\Ga{\Gamma}

\def\La{\Lambda}

\def\soc{\operatorname{soc}}
\def\Ker{\operatorname{Ker}}

\def\Im{\operatorname{Im}}

\def\Hom{\operatorname{Hom}}
\def\rad{\operatorname{rad}}

\def\Ext{\operatorname{Ext}}

\def\mod{\operatorname{mod}}
\def\Mod{\operatorname{Mod}}

\def\ind{\operatorname{ind}}

\def\prj{\operatorname{prj}}

\def\rank{\operatorname{rank}}

\def\calP{{\mathcal P}}

\def\calS{{\mathcal S}}

\newcommand\bbZ{\mathbb{Z}}
\newcommand{\sfQ}{\mathsf{Q}}

\def\op{^{\text{op}}}

\def\inv{^{-1}}

\def\iso{\cong}
\def\ds{\oplus}
\def\ox{\otimes}

\def\ovl{\overline}

\def\dsm#1,#2..#3{\bigoplus_{{#1}={#2}}^{#3}}
\def\sm#1,#2..#3{\sum_{{#1}={#2}}^{#3}}

\def\id{1\kern-.25em{\text{{\rm l}}}} 

\def\isoto{\ \raise.8ex\hbox{$^{\sim}$}\kern-.7em\hbox{$\to$}\ } 
\newcommand\Cdot{\raisebox{1pt}{\scalebox{0.4}{$\bullet$}}}
\def\down{_{\Cdot}}
\def\up{^{\Cdot}}
\def\ang#1{{\langle #1 \rangle}}
\def\ya#1{\xrightarrow{#1}}
\def\blank{\operatorname{-}}

\def\stAR#1{{}_{\text{\rm s}}\Ga_{#1}}

\def\bg{%
\family{cmr}\size{20}{12pt}\selectfont}

\def\bigzerou{%
\smash{\lower1.7ex\hbox{\bg 0}}}

\def\repr[#1;#2;#3;#4;#5]{
\left(
\begin{matrix}#1\\#2\end{matrix}
#3
\begin{matrix}#4\\#5\end{matrix}
\right)}
\def\bmat#1{\begin{bmatrix} #1 \end{bmatrix}}

\makeatother

\def\bfs{\mathbf{s}}
\def\bft{\mathbf{t}}

\def\adj{\xymatrix@1@C=10pt{\ar@{-|}[r]&}\,}

\renewcommand{\k}{\Bbbk}

\usepackage{amsmath, amsthm, amssymb, amsfonts, amscd, latexsym}

\usepackage{graphicx}
\usepackage{color}
\usepackage{enumerate}
\usepackage{multicol}
\usepackage{array}
\usepackage[all]{xy}
\usepackage{geometry}
\usepackage{enumerate}        

\def\al{\alpha}
\def\be{\beta}

\def\Ga{\Gamma}

\def\la{\lambda}
\def\La{\Lambda}
\def\ro{\rho}
\def\si{\sigma}

\def\ta{\tau}

\def\ph{\phi}

\def\cG{{\mathcal G}}

\def\ang#1{{\langle #1 \rangle}}

\def\dim{\operatorname{dim}}

\def\Ext{\operatorname{Ext}}

\def\HH{\operatorname{HH}}
\def\Hom{\operatorname{Hom}}

\def\id{\operatorname{id}}
\def\Im{\operatorname{Im}}

\def\Ker{\operatorname{Ker}}

\def\op{\mathrm{op}}

\def\rank{\operatorname{rank}}

\theoremstyle{plain} 
\newtheorem{thm}{Theorem}[section]

\newtheorem*{thm*}{Theorem}
\newtheorem*{cor*}{Corollary}
\newtheorem{lem}[thm]{Lemma}
\newtheorem{prp}[thm]{Proposition}

\theoremstyle{definition}
\newtheorem{dfn}[thm]{Definition}

\newtheorem{ntn}[thm]{Notation}
\newtheorem{rmk}[thm]{Remark}

\numberwithin{equation}{section}

\renewcommand{\r}{\mathbf{r}}
\newcommand{\Lae}{\La^{\mathrm{e}}}
\newcommand{\stAus}[1]{{}_{\mathrm{s}}\Ga_{{#1}}}

\newcommand{\inj}{\operatorname{inj}}

\title[Derived equivalence and weakly symmetric algebras]{On a derived equivalence classification of weakly symmetric algebras of Euclidean type 
}

\author{Hideto Asashiba, Ayako Itaba and Manuel Saor{\'i}n}

\begin{document}
\begin{abstract}
We consider a family $(\Lambda_\lambda)_{\la \in \k^\times}$ of weakly symmetric algebras of Euclidean type 
over an algebraically closed field $\k$, where $\k^\times:= \k\setminus \{0\}$, and for each $\lambda \in \k^\times$, $\Lambda_\lambda$ is socle equivalent to $\Lambda_1$, which is the trivial extension of a hereditary algebra of type $\tilde{A}$.
We compute their Hochschild cohomologies, which shows that $\Lambda_\lambda$
is not derived equivalent to $\Lambda_1$ if $\la \ne 1$.
This means that a derived equivalence classification of weakly symmetric algebras of Euclidean type given by Bocian--Holm--Skowro{\' n}ski is not complete.
\end{abstract}

\maketitle




\section{Introduction}

Throughout this paper, we fix an algebraically closed field $\k$,
and all vector spaces and algebras are considered over $\k$.
By an algebra, we always mean a basic, connected, finite-dimensional algebra.
We set $D:= \Hom_\k(\blank, \k)$ to be the usual $\k$-duality.
By $T(A)$, we denote the trivial extension algebra $A \ltimes D(A)$ of $A$
by the $A$-$A$-bimodule $D(A)$.
We denote by $\hat{A}$ the \emph{repetitive category} of $A$.

According to \cite{bocian2004derived}, a selfinjective algebra is said to be
\emph{of Euclidean type} if it is isomorphic to an algebra of the form
$\hat{B}/G$, where $B$ is a tilted algebra \emph{of Euclidean type}
(that is, an algebra tilted from the path algebra $\k Q$ of a quiver $Q$
with the underlying graph an extended Dynkin graph), and $G$ is an admissible  group of linear automorphisms of $\hat{B}$.
Since all weakly symmetric algebras are selfinjective (e.g.\ see \cite[page 391]{MR2894798}), we can consider weakly symmetric algebras of Euclidean type.
We cite the following from \cite[Theorem 2]{MR2130590}.

\begin{thm}
\label{thm:BS05}
Let $A$ be an algebra.
Then the following are equivalent:
\begin{enumerate}[{\rm (1)}]
\item 
$A$ is a weakly symmetric algebra of Euclidean type, and has a singular Cartan matrix.
\item 
$A$ is a symmetric algebra of Euclidean type, and has a singular Cartan matrix.
\item 
$A \iso T(B)$ for some representation-infinite tilted algebra $B$ of Euclidean type.
\end{enumerate}
\end{thm}



By \cite[Theorem 1]{bocian2004derived} (see also \cite[Theorem 2.1]{holm2006derived} for the explicit representatives), we see that the following holds.

\begin{thm}
\label{thm:der-eq-cls-std-symm-rep-inf-dom}
The following gives a complete set of representatives of
derived equivalence classes of weakly symmetric algebras
of Euclidean type with singular Cartan matrices:
\begin{multline}\label{eq:ssd}
\calS:= \{
T(p,q), T(2,2,r)^*, T(3,3,3), T(2,4,4), T(2,3,6) \mid\\
1 \le p \le q, 
2 \le r\}.
\end{multline}
\end{thm}

\begin{rmk}
An algebra is said to be \emph{standard} if it has a simply connected Galois covering.
The following follows from the main theorem of \cite{skowronski1989selfinjective}:

``A selfinjective algebra $A$ is of Euclidean type 
if and only if $A$ is standard and representation-infinite domestic.''

This statement can be  used to understand the relationship between Theorem \ref{thm:der-eq-cls-std-symm-rep-inf-dom},
\cite[Theorem 1]{bocian2004derived} and
\cite[Theorem 2.1]{holm2006derived}.
\end{rmk}


Now consider the algebras defined as follows.
These are originally constructed as {\em twisted multifold extensions} introduced in \cite{asashiba2002derived}.

\begin{dfn}
\label{dfn:La_la}
Let $\la \in \k^\times:= \k \setminus \{0\}$, and define an algebra $\La_\la$ by the quiver
$\sfQ$
\begin{equation}
\label{eq:quiver}
\begin{tikzcd}[column sep=20pt]
	& 2 && 3 & \\
	1 &&&& 4 \\
	&& 5
	\arrow["{\alpha_2}", from=1-2, to=1-4]
	\arrow["{\alpha_3}", from=1-4, to=2-5]
	\arrow["{\alpha_1}", from=2-1, to=1-2]
	\arrow["{\beta_1}"', from=2-1, to=3-3]
	\arrow["{\alpha_4}"', shift right=1.5, from=2-5, to=2-1]
	\arrow["{\beta_3}", shift left=0.5, from=2-5, to=2-1]
	\arrow["\beta_2"', from=3-3, to=2-5]
\end{tikzcd}
%
\end{equation}
with relations
$$
\begin{aligned}
&\al_4\be_1 = 0,\, \be_3\al_1 = 0,\, \al_3\be_3 = 0,\, \be_2 \al_4 = 0,\,
\al_2\al_3\al_4\al_1\al_2 = 0;\\
&\al_1\al_2\al_3\al_4 = \be_1\be_2\be_3,\,
\al_4\al_1\al_2\al_3 = \la \be_3\be_1\be_2.
\end{aligned}
$$
We set $I_\la$ to be the ideal of $\k\sfQ$ defined by the relations above.
The structures of indecomposable projective right $\La_\la$-modules are
given as follows:
\[
e_1\La_\la = 
\begin{tikzcd}[column sep=-10pt]
	& {e_1} & \\
	{\al_1} && {\be_1} \\
	{\al_1\al_2} && {\be_1\be_2} \\
	{\al_1\al_2\al_3} && {\be_1\be_2\be_3} \\
	& {\al_1\al_2\al_3\al_4}
	\arrow[from=1-2, to=2-1]
	\arrow[from=1-2, to=2-3]
	\arrow[from=2-1, to=3-1]
	\arrow[from=2-3, to=3-3]
	\arrow[from=3-1, to=4-1]
	\arrow[from=3-3, to=4-3]
	\arrow[from=4-1, to=5-2]
	\arrow[equals, from=4-3, to=5-2]  
\end{tikzcd}
\quad\quad e_2\La_\la = 
\begin{tikzcd}
	{e_2} \\
	{\alpha_2} \\
	{\alpha_2\alpha_3} \\
	{\alpha_2\alpha_3\alpha_4} \\
	{\alpha_2\alpha_3\alpha_4\alpha_1}
	\arrow[from=1-1, to=2-1]
	\arrow[from=2-1, to=3-1]
	\arrow[from=3-1, to=4-1]
	\arrow[from=4-1, to=5-1]
\end{tikzcd}
\quad\quad e_3\La_\la = 
\begin{tikzcd}
	{e_3} \\
	{\alpha_3} \\
	{\alpha_3\alpha_4} \\
	{\alpha_3\alpha_4\alpha_1} \\
	{\alpha_3\alpha_4\alpha_1\alpha_2}
	\arrow[from=1-1, to=2-1]
	\arrow[from=2-1, to=3-1]
	\arrow[from=3-1, to=4-1]
	\arrow[from=4-1, to=5-1]
\end{tikzcd}
\]

\[
e_4\La_\la = 
\begin{tikzcd}[column sep=-10pt]
	& {e_4} & \\
	{\al_4} && {\be_3} \\
	{\al_4\al_1} && {\be_3\be_1} \\
	{\al_4\al_1\al_2} && {\lambda\be_3\be_1\be_2} \\
	& {\al_4\al_1\al_2\al_3}
	\arrow[from=1-2, to=2-1]
	\arrow[from=1-2, to=2-3]
	\arrow[from=2-1, to=3-1]
	\arrow[from=2-3, to=3-3]
	\arrow[from=3-1, to=4-1]
	\arrow[from=3-3, to=4-3, "\la\inv"]
	\arrow[from=4-1, to=5-2]
	\arrow[equals, from=4-3, to=5-2] 
\end{tikzcd}
\quad\quad e_5\La_\la = 
\begin{tikzcd}
	{e_5} \\
	{\beta_2} \\
	{\beta_2\beta_3} \\
	{\beta_2\beta_3\beta_1}
	\arrow[from=1-1, to=2-1]
	\arrow[from=2-1, to=3-1]
	\arrow[from=3-1, to=4-1]
\end{tikzcd}
\]
\end{dfn}

We denote by $\stAR{\La}$ the stable Auslander--Reiten quiver
(stable AR-quiver for short) of an algebra $\La$.
In section 2,  we will prove the following.

\begin{prp}[Proposition \ref{prp:fundamental-prop}]
\label{prp:prop-La-la}
For each $\la \in \k^\times$, $\La_\la$ is a  weakly symmetric algebra of Euclidean type with a singular Cartan matrix.
\end{prp}

Hence by Theorem \ref{thm:der-eq-cls-std-symm-rep-inf-dom},
for each $\la \in k^\times$,
$\La_\la$ should be derived equivalent to a unique algebra $A_\la \in \calS$
in \eqref{eq:ssd}.
Then $A_\la$ and $\La_\la$ are stably equivalent, and hence
$$
\stAR{A_\la} \iso \stAR{\La_\la}
$$
as valued translation quivers.
In Proposition \ref{prp:st-AR-same}, we will show that
$\stAus{\La_\la} \iso \stAus{\La_1}$ for all $\la \in \k^\times$.
Then we have
$$
\stAus{A_\la} \iso \stAus{\La_1}.
$$
Here, note that both $A_\la$ and $T(2,3)\ (\iso \La_1)$ are in $\calS$.
Since for any algebras $A, B \in \calS$, $\stAus{A} \iso \stAus{B}$
if and only if $A = B$,
we see that $A_\la = T(2,3) \iso \La_1$.
Hence $\La_\la$ should be derived equivalent to $\La_1$ for all $\la \in \k^\times$.

In this paper, we compute the Hochschild cohomology of $\La_\la$ to check
this statement, and we obtain the following.

\begin{thm*}[Theorem \ref{thm:dim-HH}]
For each $n = 0, 1, 2$,
the $\k$-dimension of $n$-th Hochschild cohomology group
of $\La_\la$ is given as follows .
$$
\begin{aligned}
\dim_\k \HH^0(\La_\la) &= 6.\\
\dim_\k \HH^1(\La_\la) &= 2.\\
\dim_\k \HH^2(\La_\la) &=
\begin{cases}
1 & \text{ if } \la \ne 1,\\
2 & \text{ if } \la = 1.
\end{cases}
\end{aligned}
$$
\end{thm*}

Since the Hochschild cohomology groups are derived invariant
(\cite{rickard1991derived}), we obtain the following.

\begin{cor*}[Theorem \ref{thm:not-der-eq}]
Let $\la \in \k^\times$.
If $\la \ne 1$, 
then $\La_{\la}$ is not derived equivalent to $\La_{1}$. 
\end{cor*}

As a consequence, Theorem \ref{thm:der-eq-cls-std-symm-rep-inf-dom} (hence also \cite[Theorem 1]{bocian2004derived}) is not correct.
We will see that this error is caused by Theorem \ref{thm:BS05} by showing the following statement, which together with Proposition \ref{prp:prop-La-la} makes a counterexample of Theorem \ref{thm:BS05} for $\la \ne 1$. 

\begin{prp}
[Proposition \ref{prp:symmetric-iff-la=1}]
For any $\la \in \k^\times$, 
$\La_\la$ is symmetric if and only if $\la = 1$.
\end{prp}

The paper is organized as follows.
In Sect.\ 2, we prove that for each $\la \in \k^\times$,
$\La_\la$ is a weakly symmetric algebra
of Euclidean type, and with a singular Cartan matrix.
Moreover, we show for each $\la \in \k^\times$ that
$\La_\la$ is symmetric if and only if $\la = 1$, and that
$\La_\la$ and $\La_1$ have a common stable AR-quiver.

In Sect.\ 3, we compute Hochschild cohomologies of $\La_\la$, and in Appendix A, we give a proof of Proposition \ref{prp:stAR-iso-AR(over-socle)'-as-tr-qv} for the benefit of the reader.
In Appendix B, we extend Proposition 2.8 and Theorem 2.9 in \cite{green2004projective}
by Green--Sashall, which is used in this paper to compute a projective bimodule resolution of a bimodule that is projective as a right module.

\section*{Acknowledgments}

This project originated in discussions between HA and MS during HA's visit to Murcia in March 2015. These discussions led to the discovery of the family $(\Lambda_\lambda)_{\lambda\in k^\times}$ and the observation that $\Lambda_\lambda$ is symmetric if and only if $\lambda=1$. This in turn led to the conjecture that $\Lambda_\lambda$ and $\Lambda_\mu$ are derived equivalent if and only if $\lambda=\mu$.

AI joined the project in September 2015 to undertake the Hochschild cohomology computations. After a long interruption, HA and AI resumed the project in October 2025 and got the counterexample mentioned above. 

The result was announced at a conference ``Representation Theory of Rings and Triangulated Categories'' held at Okayama University of Science in honor of Professor Kiriko Kato's 60-th birthday on October 3, 2026.  AI would like to thank the organizers of the conference for giving an opportunity to give a talk
and some participants for valuable comments.

A proof of the derived equivalence conjecture will be presented in a separate paper by HA.

For some of the proofs an extension of a result of Green and Snashall, given in a final appendix written by HA in 2016, is used. HA thanks Edward Green for his valuable suggestions during their email correspondence.

HA is partially supported by JSPS Grant-in-Aid for Challenging Exploratory Research 25610003, Scientific Research (B) 25287001, (C) 18K03207, (C) 25K06922, and by Osaka Central Advanced Mathematical Institute
(MEXT Promotion of Distinctive Joint Research Center Program JPMXP0723833165).

AI is supported by JSPS Grant-in-Aid for Scientific Research (C) 24K06653.

MS is supported by the grant PID2024-155576NB-I00, funded by the MICIU/AEI and FEDER of the UE, and the project 2024/PI/22, funded by the Fundaci\'on 'S\'eneca' of Murcia.

\color{black}
\section{Properties of $\La_\la$}

In this section, we prove fundamental properties of our algebras $\La_\la$
\ ($\la \in \k^\times$).

\begin{prp}
\label{prp:fundamental-prop}
For any $\la \in \k^\times$,
$\La_\la$ is weakly symmetric algebra of Euclidean type with a singular Cartan matrix.
\end{prp}

\begin{proof}
It is obvious that $\La_\la$ is weakly symmetric by looking at the structures of projective indecomposable right (resp.\ left) modules
(see Definition \ref{dfn:La_la}). 

The Cartan matrix of $\La_\la$ is given by
$$
\bmat{
2&1&1&2&1\\
1&2&1&1&0\\
1&1&2&1&0\\
2&1&1&2&1\\
1&0&0&1&2
},
$$
and its determinant is 0.

Recall that a twisted multifold extension $T^n_\ph(A)$ of an algebra $A$ was defined as the (classical) orbit category
$$
T^n_\ph(A):= \hat{A}/\ang{\hat{\ph}\nu_A^n},
$$
where $n$ is a positive integer, and $\ph$ is an automorphism of $A$,
$\nu_A$ is the Nakayama automorphism of $\hat{A}$, and
$\hat{\ph}$ is an automorphism of $\hat{A}$ induced by $\ph$ canonically
(see \cite{asashiba2002derived} for details).
We use this construction.

Let $B$ be the hereditary algebra of type $\widetilde{A}_{2,3}$
defined by the quiver $\sfQ'$ obtained from $\sfQ$ by removing the arrows $\al_4$ and $\be_3$.
For each $\la \in \k^\times$, define an automorphism $\ph_\la$ of $B$ by
fixing all vertices of $\sfQ'$ and by 
$$
\ph_\la(\al):= \begin{cases}
\la\al & \text{if $\al = \al_1$}\\
\al &\text{otherwise}
\end{cases}
$$
for all arrows $\al$ in $\sfQ'$.

Then it is easy to see by the computation of orbit category stated in
\cite{asashiba2011generalization} (by taking its skeleton) that we have
\begin{equation}
\label{eq:La_la-as-tme}
\La_\la \iso T^1_{\ph_\la}(B) = \hat{B}/\ang{g_\la},
\end{equation}
where we put $g_\la:= \widehat\ph_\la\nu_B$.
Since $B$ is hereditary of type $\widetilde A_{2,3}$,
it is a tilted algebra of Euclidean type.
Since $\ph_\la$ fixes the vertices of $\sfQ'$, we have $g_\la(x^{[i]})=x^{[i+1]}$ for all $x \in \sfQ'_0$ and  $i \in \bbZ$.
Thus $\ang{g_\la}$ is an admissible infinite cyclic group of automorphisms of $\hat B$. The isomorphism \eqref{eq:La_la-as-tme}
therefore shows that $\La_\la$ is of Euclidean type.
\end{proof}

The following can be directly verified by an elementary argument, but
here we show it by applying \cite[Theorem 2.1 and Lemma 1.3]{ohnuki2000automorphisms}.

\begin{prp}
\label{prp:symmetric-iff-la=1}
For any $\la \in \k^\times$, 
$\La_\la$ is symmetric if and only if $\la = 1$.
\end{prp}

\begin{proof}
Let $\la \in \k^\times$, and $B$ be the hereditary algebra of type $\tilde{A}_{2,3}$ as in the proof of Proposition \ref{prp:fundamental-prop}.
By \cite[Theorem 2.1]{ohnuki2000automorphisms}, we have an isomorphism
$$
\La_\la \iso \hat{B}/\ang{\widehat{\ph}_\la\nu_B} \iso B \ltimes (DB)_{\ph_\la}.
$$
Assume that $\La_\la$ is symmetric.
Then by \cite[Lemma 1.3]{ohnuki2000automorphisms},
$\ph_\la$ is an inner automorphism of $B$, that is,
$\ph_\la = u\inv (\blank) u$ for some unit $u \in B$.
Since $\ph_\la$ fixes all vertices $i \in \sfQ'_0 =\{1,\dots,5\}$, we have $u\inv e_i u = e_i$. Thus
\begin{equation}
\label{eq:u-commutes-e_i}
e_i u  = u e_i
\end{equation}
for all $i = 1, \dots, 5$.
Since $1_B = e_1 + \cdots + e_5$ and $e_iBe_i = \k e_i$ for all $i = 1,\dots, 5$, we have 
$$
u = \sum_{i1}^5 u e_i \overset{\eqref{eq:u-commutes-e_i}}{=} \sum_{i1}^5 e_i u e_i = \sum_{i1}^5 c_i e_i
$$
for some $c_1, \dots, c_5 \in \k$.
Note that since $u$ is a unit of $B$, $0 \ne ue_i = c_i e_i$, and hence $c_i \ne 0$ for all $i=1,\dots, 5$.
Put $p:= \al_1 \al_2 \al_3$ and $q:= \be_1 \be_2$.
Then both are paths from 1 to 4.
Therefore,
$\ph_\la(p) = \la p$ shows that $\la p = u\inv p u = c_1\inv c_4 p$.
Moreover, $\ph_\la(q) = q$ shows that $q = u\inv q u = c_1\inv c_4 q$.
Hence $\la = c_1\inv c_4 = 1$.

Conversely, if $\la = 1$, then $\ph_\la = \ph_1 = \id_B$, and
$\La_1 = B \ltimes DB$ is symmetric.
\end{proof}

We will show that
algebras $\La_\la$ have a common stable AR-quiver for all 
$\la \in \k^\times$ using the following.

\begin{prp}
\label{prp:stAR-iso-AR(over-socle)'-as-tr-qv}
Let $\La$ be a selfinjective algebra, which is not semisimple. 
Then we have 
$$
\stAR{\La} \iso \Ga_{\La/\soc\La}'
$$
as valued translation quivers, where $\Ga_{\La/\soc\La}'$ is constructed
from $\Ga_{\La/\soc\La}$ by adding the translations
$$
D((\La/\soc \La)e) \dashleftarrow e(\La/\soc \La)
$$
for all primitive idempotents $e$ of $\La/\soc \La$.
\end{prp}

This is proved in Appendix (see Proposition \ref{prp:stAR-again}).

\begin{prp}
\label{prp:st-AR-same}
For any $\la \in \k^\times$, 
$\stAus{\La_\la} \iso \stAus{\La_1}$ as valued translation quivers.
\end{prp}

\begin{proof}
This follows from the following isomorphisms of translation quivers:
$$
\stAus{\La_\la} \iso \Ga'_{\La_\la/\soc\La_\la} \iso \Ga'_{\La_1/\soc\La_1}\iso \stAus{\La_1},
$$
where the first and the third isomorphisms follow by Proposition \ref{prp:stAR-iso-AR(over-socle)'-as-tr-qv}, and
the second isomorphism follows from the fact that $\La_\la/\soc \La_\la \iso \La_1/\soc\La_1$ for all $\la \in \k^\times$
(this follows by definition of $\La_\la$).
%
%
\end{proof}













\section{Hochschild cohomology groups}
Let $\La=\k\sfQ/I$ be a finite-dimensional algebra 
over an algebraically closed field $\k$ for a
finite quiver $\sfQ=(\sfQ_{0},\sfQ_{1},\bfs, \bft)$,
where $\sfQ_{0}$ (resp.\ $\sfQ_{1}$) is a set of vertices
(resp.\ the set of arrows) of $\sfQ$,  
$\bfs$ (resp. $\bft$) is a map sending each arrow to its source
(resp.\ target), 
and $I$ is an admissible ideal of $\k\sfQ$. 
For the general notation, we refer to \cite{assem2006elements}. 
Tensor products over $\k$ are simply denoted by $\ox$.

\subsection{Definition and computation of Hochschild cohomologies}

Let $\La$ be a finite-dimensional algebra over a field $\k$ of the form
$\k \sfQ/I$ for some finite quiver and an admissible ideal $I$ of $\k\sfQ$.
We set
$$
\La^{\rm e}:=\La^{\op}\otimes \La
$$
to be the \emph{enveloping algebra} of $\La$. 
Note that any $\La$-bimodule $M$ is regarded as a right $\Lae$-module
by setting
\begin{equation}
\label{eq:bimod=r_mod_env}
m\cdot(a \ox b) = amb
\end{equation}
for all $m \in M$ and $a, b \in \La$.
Conversely, any right $\Lae$-module $M$ is regarded as a $\La$-bimodule
by setting the right hand side of \eqref{eq:bimod=r_mod_env}
to be the left hand side of \eqref{eq:bimod=r_mod_env}
for all $m \in M$ and $a, b \in \La$.
These show that
the category $\Mod \La^{\rm e}$ of the right $\La^{\rm e}$-modules is isomorphic to the category of the $\La$-bimodules.
By this isomorphism, we identify these categories.

Let us recall the definition of the Hochschild cohomology groups of $\La$.
For a non-negative integer $n$, 
the $n$-th \emph{Hochschild cohomology group} $\HH^{n}(\La)$ of $\La$ is defined by 
$$
\HH^{n}(\La):=\Ext_{\La^{\rm e}}^{n}(\La,\La)\ (n\geq 0).
$$

We remark that the low-dimensional Hochschild cohomology groups have the following meanings:
\begin{itemize}
\item $\HH^{0}(\La)$ is the center $Z(\La)$ of $\La$. 

\item $\HH^{1}(\La)$ is the space of derivations modulo inner derivations. 
Here, a \emph{derivation} is a $\k$-linear map $f:\La \rightarrow \La$ 
such that $f(ab)=af(b)+f(a)b$ for all $a,b\in \La$.
Furthermore,
a derivation $f:\La \rightarrow \La$ is called an \emph{inner derivation} 
if there is some $x\in \La$ such that 
$f(a)=ax-xa$ for all $a\in \La$. 

\item $\HH^{2}(\La)$ gives information on the infinitesimal deformation of the algebra $\La$. 
\end{itemize}

One important property of the Hochschild cohomology is in its invariance under Morita equivalence, derived equivalence 
(\cite{rickard1991derived}). 

To calculate the Hochschild cohomology, 
we start by recalling the construction of sets 
$\cG^{n}$ $(n\geq 0)$ in \cite{green2001minimal}. 
We set $\cG^{0}:= \sfQ_0$, the set of vertices of $\sfQ$, 
$\cG^{1}:= \sfQ_1$, the set of arrows of $\sfQ$, 
and $\cG^{2}$ a minimal set of uniform generators of $I$. 
An element of a path algebra of a quiver is said to be \emph{uniform}
if it is a linear combination of parallel paths.
In \cite{green2001minimal}, Green, Solberg and Zacharia proved that, for each $n\geq 3$, 
we have a set $\mathcal{G}^{n}$ consisting of uniform elements in $\k\sfQ$ 
such that there is a minimal projective resolution $(P\up,d\up)$ 
of the right $\La$-module $\La/\rad\La$ satisfying 
the following conditions (a), (b) and (c): 
\begin{enumerate}[(a)]
\item For each $n\geq 0$, 
$P^{n}=\bigoplus_{g\in \cG^{n}} \bft(g) \La$. 

\item For each $g\in \cG^{n}$ with $n \ge 2$, 
there exist unique families
$r_{h} \in \bft(h)(\k \sfQ)\bft(g)$ 
and 
$s_{l} \in \bft(l) I \bft(g)$
such that
$$
g=\sum_{h\in {\cG}^{n-1}} h r_{h}=\sum_{l\in {\cG}^{n-2}} l s_{l}.
$$

\item For each $n \geq 1$, $d^{n}:P^{n}\rightarrow P^{n-1}$ 
is defined by 
$$
d^{n}(t(g)a)= \left(\sum_{h\in {\cG}^{n-1}}r_{h}\right) t(g)
a
$$ 
for all $g\in {\cG}^{n}$ and $a\in \La$, 
where $r_{h}$ are given in (b) above. 
\end{enumerate}
By constructing the sets $\mathcal{G}^{n}$ $(n\geq 0)$ 
for the right $\La$-module 
$\La/{\rm rad}\,\La$, 
we can also give a minimal projective resolution 
$(Q\up,\partial\up)$ of the right 
$\La^{\rm e}$-module $\La$  by 
\color{black}
applying Theorem \ref{thm:ext-GS} 
in Appendix B, which is an extension of Proposition 2.8 and Theorem 2.9 in \cite{green2004projective}
by Green--Snashall, where $Q\up$ is given by
$Q^{n}:=\bigoplus_{g^{n}\in {\cG}^{n}} \La \bfs(g^{n})\otimes \bft(g^{n})\La$ 
for all $n\geq 0$. 
\color{black}

\subsection{Calculation of Hochschild cohomologies}


\begin{thm}
\label{thm:dim-HH}
For each $n = 0, 1, 2$,
the $\k$-dimension of $n$-th Hochschild cohomology group
of $\La_\la$ is given as follows\textup{:}

$$
\begin{aligned}
\dim_\k \HH^0(\La_\la) &= 6,\\
\dim_\k \HH^1(\La_\la) &= 2,\\
\dim_\k \HH^2(\La_\la) &=
\begin{cases}
1 & \text{ if } \la \ne 1,\\
2 & \text{ if } \la = 1.
\end{cases}
\end{aligned}
$$
\end{thm}

It is known that Hochschild cohomology is invariant under derived equivalence.
Using the above dimensional formula, 
we have the following consequence. 

\begin{thm}
\label{thm:not-der-eq}
Let $\la \in \k^\times$.
If $\la \ne 1$, 
then $\La_{\la}$ is not derived equivalent to $\La_{1}$. 
\end{thm}

\subsection*{Proof of Theorem \ref{thm:dim-HH}}
For simplicity we fix an arbitrary $\la \in \k^\times$, and simply put
$I:= I_\la$ and $\La:= \La_\la = \k \sfQ/I$.
To calculate the lower dimensional Hochschild cohomology group of $\La$, 
we define the set $\cG^{n}$ $(n=0,1,2,3)$ in \cite{green2001minimal}: 
\begin{align*}
\cG^{0}:= & \sfQ_{0}
        =\{g_{i}^{0}:=e_{i}\mid i=1,2,3,4,5\}, \\
\cG^{1}:= & \sfQ_{1}
          =\{
          g_{1}^{1}:=g_{1}^{0}\al_{1}=\al_{1},\,
          g_{2}^{1}:=g_{2}^{0}\al_{2}=\al_{2},\,\\ &
          g_{3}^{1}:=g_{3}^{0}\al_{3}=\al_{3},\,
          g_{4}^{1}:=g_{4}^{0}\al_{4}=\al_{4},\,\\ &
          g_{5}^{1}:=g_{1}^{0}\be_{1}=\be_{1},\,
          g_{6}^{1}:=g_{5}^{0}\be_{2}=\be_{2},\,
          g_{7}^{1}:=g_{4}^{0}\be_{3}=\be_{3}
          \}, \\
\cG^{2}:=&
         \{
         g_{1}^{2}:=g_{1}^{1}\al_{2}\al_{3}\al_{4}-g_{5}^{1}\be_{2}\be_{3}
                   =\al_{1}\al_{2}\al_{3}\al_{4}-\be_{1}\be_{2}\be_{3},\,\\ &
         g_{2}^{2}:=g_{2}^{1}\al_{3}\al_{4}\al_{1}\al_{2}
                   =\al_{2}\al_{3}\al_{4}\al_{1}\al_{2},\,\\ &
         g_{3}^{2}:=g_{3}^{1}\be_{3}=\al_{3}\be_{3},\,
         g_{4}^{2}:=g_{4}^{1}\be_{1}=\al_{4}\be_{1},\, \\ &
          g_{5}^{2}:=g_{4}^{1}\al_{1}\al_{2}\al_{3}-\la g_{7}^{1}\be_{1}\be_{2}
                   =\al_{4}\al_{1}\al_{2}\al_{3}-\la\be_{3}\be_{1}\be_{2},\,\\ &
         g_{6}^{2}:=g_{7}^{1}\al_{1}=\be_{3}\al_{1},\,
         g_{7}^{2}:=g_{6}^{1}\al_{4}=\be_{2}\al_{4},\,
         \}, \\
\cG^{3}:=&
         \{
         g_{1}^{3}:=g_{1}^{2}\al_{1}\al_{2},\,
         g_{2}^{3}:=\la g_{1}^{2}\be_{1}\be_{2},\,
         g_{3}^{3}:=g_{2}^{2}\al_{3},\,\\ &
         g_{4}^{3}:=g_{3}^{2}\al_{1},\,
         g_{5}^{3}:=-g_{4}^{2}\be_{2}\be_{3}+g_{5}^{2}\al_{4},\,
         g_{6}^{3}:=g_{5}^{2}\be_{3}+\la g_{6}^{2}\al_{2}\al_{3}\al_{4},\,
         g_{7}^{3}:=g_{7}^{2}\be_{1}
         \}. 
\end{align*}
Using these sets $\cG^{n}$ ($n=0,1,2,3$), 
we give a starting part of a minimal projective resolution 
$(P\up,d\up)$ of ($\La/\rad \La$) 
as a right $\La$-module. 
For each $n=0,1,2,3$, we define the projective 
$\La$-module $P^{n}$
by 
$$
    P^{n}:=\bigoplus_{g \in \cG^{n}}\bft(g) \La,
$$
and the map $d^{n}$: $P^{n}\rightarrow P^{n-1}$ 
as follows: 
  We define $d^{0}:P^{0}\rightarrow \La/\rad \La$ to be 
      the canonical projection
  and $d^{1}:P^{1}\rightarrow P^{0}$ 
  to be the right $\La$-homomorphism determined by 
\begin{align*}
             &\bft(g_{1}^{1})
             \longmapsto \bft(g_{1}^{0})\al_{1}, \quad
             \bft(g_{2}^{1})
             \longmapsto \bft(g_{2}^{0})\al_{2}, \quad
             \bft(g_{3}^{1})
             \longmapsto \bft(g_{3}^{0})\al_{3}, \quad \\
             & \bft(g_{4}^{1})
             \longmapsto \bft(g_{4}^{0})\al_{4},
             \bft(g_{5}^{1})
             \longmapsto \bft(g_{1}^{0})\be_{1}, \quad
             \bft(g_{6}^{1})
             \longmapsto \bft(g_{5}^{0})\be_{2}, \quad
             \bft(g_{7}^{1})
             \longmapsto \bft(g_{4}^{0})\be_{3}. 
\end{align*}
\noindent
We define further $d^{2}:P^{2}\rightarrow P^{1}$ 
to be the right $\La$-homomorphism determined by 
 \begin{align*}
             &\bft(g_{1}^{2})
             \longmapsto \bft(g_{1}^{1})\al_{2}\al_{3}\al_{4}
             -\bft(g_{5}^{1})\be_{2}\be_{3}, \quad
             \bft(g_{2}^{2})
             \longmapsto \bft(g_{2}^{1})\al_{3}\al_{4}\al_{1}\al_{2}, \quad \\
             & \bft(g_{3}^{2}) 
             \longmapsto \bft(g_{3}^{1})\be_{3}, 
             \bft(g_{4}^{2})
             \longmapsto \bft(g_{4}^{1})\be_{1},\quad
             \bft(g_{5}^{2})
             \longmapsto \bft(g_{4}^{1})\al_{1}\al_{2}\al_{3}
             -\la \bft(g_{7}^{1})\be_{1}\be_{2}, \quad \\
             & \bft(g_{6}^{2})
             \longmapsto \bft(g_{7}^{1})\al_{1}, \quad
             \bft(g_{7}^{2})
             \longmapsto \bft(g_{6}^{1})\al_{4},
\end{align*}
\noindent
and $d^{3}:P^{3}\rightarrow P^{2}$ to be the right $\La$-homomorphism determined by 
 \begin{align*}
             &\bft(g_{1}^{3})
             \longmapsto \bft(g_{1}^{2})\al_{1}\al_{2}, \quad
             \bft(g_{2}^{3})
             \longmapsto \la \bft(g_{1}^{2})\be_{1}\be_{2}, \quad \\
             & \bft(g_{3}^{3})
             \longmapsto \bft(g_{2}^{2})\al_{3}, \quad 
             \bft(g_{4}^{3})
             \longmapsto \bft(g_{3}^{2})\al_{1},
             \bft(g_{5}^{3})
             \longmapsto -\bft(g_{4}^{2})\be_{2}\be_{3}
             +\bft(g_{5}^{2})\al_{4}, \quad \\
             & \bft(g_{6}^{3})
             \longmapsto \bft(g_{5}^{2})\be_{3}
             +\la \bft(g_{6}^{2})\al_{2}\al_{3}\al_{4}, \quad
             \bft(g_{7}^{3})
             \longmapsto \bft(g_{7}^{2})\be_{1}. 
\end{align*}

Moreover, using these sets $\cG^{n}$ ($n=0,1,2,3$), 
we also give the beginning of a minimal projective resolution 
$(Q\up,\partial\up)$ of $\La$ 
as a right $\La^{\rm e}$-module. 
First, we give the definition of the projective 
$\La^{\rm e}$-module $Q^{n}$ for $n=0,1,2,3$. 
The projective $\La^{\rm e}$-module $Q^{n}$ is defined by 
$$
    Q^{n}:=\bigoplus_{g \in \cG^{n}}\La \bfs(g)\otimes \bft(g)\La
    \quad\text{ for }n=0,1,2,3.
$$

Next, we define the map $\partial^{n}$: $Q^{n}\rightarrow Q^{n-1}$ 
as follows: 
  We define $\partial^{0}:Q^{0}\rightarrow \La$ to be the multiplication map
  and $\partial^{1}:Q^{1}\rightarrow Q^{0}$ 
  to be the $\La^{\rm e}$-homomorphism determined by 
 \begin{align*}
  \begin{cases}
             \bfs(g_{1}^{1})\otimes \bft(g_{1}^{1})
             \longmapsto (\bfs(g_{1}^{0})\otimes \bft(g_{1}^{0}))\al_{1}
             -\al_{1}(\bfs(g_{2}^{0})\otimes \bft(g_{2}^{0}))
             , \\
             \bfs(g_{2}^{1})\otimes \bft(g_{2}^{1})
             \longmapsto (\bfs(g_{2}^{0})\otimes \bft(g_{2}^{0}))\al_{2}
             -\al_{2}(\bfs(g_{3}^{0})\otimes \bft(g_{3}^{0}))
             , \\
             \bfs(g_{3}^{1})\otimes \bft(g_{3}^{1})
             \longmapsto (\bfs(g_{3}^{0})\otimes \bft(g_{3}^{0}))\al_{3}
             -\al_{3}(\bfs(g_{4}^{0})\otimes \bft(g_{4}^{0}))
             , \\
             \bfs(g_{4}^{1})\otimes \bft(g_{4}^{1})
             \longmapsto (\bfs(g_{4}^{0})\otimes \bft(g_{4}^{0}))\al_{4}
             -\al_{4}(\bfs(g_{1}^{0})\otimes \bft(g_{1}^{0}))
             , \\
             \bfs(g_{5}^{1})\otimes \bft(g_{5}^{1})
             \longmapsto (\bfs(g_{1}^{0})\otimes \bft(g_{1}^{0}))\be_{1}
             -\be_{1}(\bfs(g_{5}^{0})\otimes \bft(g_{5}^{0}))
             , \\
             \bfs(g_{6}^{1})\otimes \bft(g_{6}^{1})
             \longmapsto (\bfs(g_{5}^{0})\otimes \bft(g_{5}^{0}))\be_{2}
             -\be_{2}(\bfs(g_{4}^{0})\otimes \bft(g_{4}^{0}))
             , \\
             \bfs(g_{7}^{1})\otimes \bft(g_{7}^{1})
             \longmapsto (\bfs(g_{4}^{0})\otimes \bft(g_{4}^{0}))\be_{3}
             -\be_{3}(\bfs(g_{1}^{0})\otimes \bft(g_{1}^{0})).
 \end{cases}
\end{align*}
\noindent
Also, $\partial^{2}:Q^{2}\rightarrow Q^{1}$ 
is the $\La^{\rm e}$-homomorphism determined by 
\begin{align*}
 \begin{cases}
     \bfs(g_{1}^{2})\otimes \bft(g_{1}^{2})\longmapsto 
         &\al_{1}\al_{2}\al_{3}(\bfs(g_{4}^{1})\otimes \bft(g_{4}^{1}))
         +\al_{1}\al_{2}(\bfs(g_{3}^{1})\otimes \bft(g_{3}^{1}))\al_{4}\\
         &\quad+\al_{1}(\bfs(g_{2}^{1})\otimes \bft(g_{2}^{1}))\al_{3}\al_{4}
         +(\bfs(g_{1}^{1})\otimes \bft(g_{1}^{1}))\al_{2}\al_{3}\al_{4}
         )\\
         &\quad 
         -(\be_{1}\be_{2}(\bfs(g_{7}^{1})\otimes \bft(g_{7}^{1})
         +\be_{1}(\bfs(g_{6}^{1})\otimes \bft(g_{6}^{1})\be_{3} \\
        &\quad  +(\bfs(g_{5}^{1})\otimes \bft(g_{5}^{1})\be_{2}\be_{3}),\\
     \bfs(g_{2}^{2})\otimes \bft(g_{2}^{2})\longmapsto 
         &\al_{2}\al_{3}\al_{4}\al_{1}(\bfs(g_{2}^{1})\otimes \bft(g_{2}^{1}))
         +\al_{2}\al_{3}\al_{4}(\bfs(g_{1}^{1})\otimes \bft(g_{1}^{1}))\al_{2}\\
         &\quad+\al_{2}\al_{3}(\bfs(g_{4}^{1})\otimes \bft(g_{4}^{1}))\al_{1}\al_{2}
         +\al_{2}(\bfs(g_{3}^{1})\otimes \bft(g_{3}^{1}))\al_{4}\al_{1}\al_{2}\\
         &\quad+(\bfs(g_{2}^{1})\otimes \bft(g_{2}^{1}))\al_{3}\al_{4}\al_{1}\al_{2},\\
     \bfs(g_{3}^{2})\otimes \bft(g_{3}^{2})\longmapsto 
         &\al_{3}(\bfs(g_{7}^{1})\otimes \bft(g_{7}^{1}))
         +(\bfs(g_{3}^{1})\otimes \bft(g_{3}^{1}))\be_{3},\\
     \bfs(g_{4}^{2})\otimes \bft(g_{4}^{2})\longmapsto 
         &\al_{4}(\bfs(g_{5}^{1})\otimes \bft(g_{5}^{1}))
         +(\bfs(g_{4}^{1})\otimes \bft(g_{4}^{1})\be_{1}),\\
     \bfs(g_{5}^{2})\otimes \bft(g_{5}^{2})\longmapsto 
         &(\al_{4}\al_{1}\al_{2}(\bfs(g_{3}^{1})\otimes \bft(g_{3}^{1}))
         +\al_{4}\al_{1}(\bfs(g_{2}^{1})\otimes \bft(g_{2}^{1}))\al_{3} \\
         &+\al_{4}(\bfs(g_{1}^{1})\otimes \bft(g_{1}^{1}))\al_{2}\al_{3} 
         +(\bfs(g_{4}^{1})\otimes \bft(g_{4}^{1}))\al_{1}\al_{2}\al_{3})\\
         &\quad 
         -\la(\be_{3}\be_{1}(\bfs(g_{6}^{1}))\otimes \bft(g_{6}^{1})
         +\be_{3}(\bfs(g_{5}^{1}))\otimes \bft(g_{5}^{1})\be_{2} \\
         &\quad +(\bfs(g_{7}^{1})\otimes \bft(g_{7}^{1}))\be_{1}\be_{2}),\\
      \bfs(g_{6}^{2})\otimes \bft(g_{6}^{2})\longmapsto 
         &\be_{3}(\bfs(g_{1}^{1})\otimes \bft(g_{1}^{1}))
         +(\bfs(g_{7}^{1})\otimes \bft(g_{7}^{1}))\al_{1},\\
      \bfs(g_{7}^{2})\otimes \bft(g_{7}^{2})\longmapsto 
         &\be_{2}(\bfs(g_{4}^{1})\otimes \bft(g_{4}^{1}))
         +(\bfs(g_{6}^{1})\otimes \bft(g_{6}^{1}))\al_{4}. 
 \end{cases}
\end{align*}
\noindent
Finally, $\partial^{3}:Q^{3}\rightarrow Q^{2}$ 
is the $\La^{\rm e}$-homomorphism determined by 
\begin{align*}
\begin{cases}
     \bfs(g_{1}^{3})\otimes \bft(g_{1}^{3})\longmapsto 
         & (\bfs(g_{1}^{2}))\otimes \bft(g_{1}^{2})\al_{1}\al_{2}
         -(\al_{1}(\bfs(g_{2}^{2})\otimes \bft(g_{2}^{2})) \\
         & \quad +(\be_{1}\be_{2}(\bfs(g_{6}^{2})\otimes \bft(g_{6}^{2}))\al_{2}),\\
     \bfs(g_{2}^{3})\otimes \bft(g_{2}^{3})\longmapsto 
         &\be_{1}(\bfs(g_{7}^{2})\otimes \bft(g_{7}^{2}))\al_{1}\al_{2}\al_{3}
         -\la \al_{1}\al_{2}\al_{3}(\bfs(g_{4}^{2})\otimes \bft(g_{4}^{2}))\be_{2}\\
         &\quad+\la(\bfs(g_{1}^{2})\otimes \bft(g_{1}^{2}))\be_{1}\be_{2}
         -\be_{1}\be_{2}(\bfs(g_{5}^{2})\otimes \bft(g_{5}^{2})),\\
     \bfs(g_{3}^{3})\otimes \bft(g_{3}^{3})\longmapsto 
         &(\bfs(g_{2}^{2})\otimes \bft(g_{2}^{2}))\al_{3}
         -\al_{2}\al_{3}(\bfs(g_{5}^{2})\otimes \bft(g_{5}^{2})) \\
         & \quad -\la \al_{2}(\bfs(g_{3}^{2})\otimes \bft(g_{3}^{2}))\be_{1}\be_{2},\\
     \bfs(g_{4}^{3})\otimes \bft(g_{4}^{3})\longmapsto 
         &(\bfs(g_{3}^{2})\otimes \bft(g_{3}^{2}))\al_{1}
         -\al_{3}(\bfs(g_{6}^{2})\otimes \bft(g_{6}^{2})),\\
     \bfs(g_{5}^{3})\otimes \bft(g_{5}^{3})\longmapsto
         &-\al_{4}(\bfs(g_{1}^{2})\otimes \bft(g_{1}^{2}))
         -(\bfs(g_{4}^{2})\otimes \bft(g_{4}^{2}))\be_{2}\be_{3}
         +(\bfs(g_{5}^{2})\otimes \bft(g_{5}^{2}))\al_{4}\\
         &\quad+\la \be_{3}\be_{1}(\bfs(g_{7}^{2})\otimes \bft(g_{7}^{2})),\\
      \bfs(g_{6}^{3})\otimes \bft(g_{6}^{3})\longmapsto 
         &-\la \be_{3}(\bfs(g_{1}^{2})\otimes \bft(g_{1}^{2}))
         -\al_{4}\al_{1}\al_{2}(\bfs(g_{3}^{2})\otimes \bft(g_{3}^{2}))\\
         &\quad +(\bfs(g_{5}^{2})\otimes \bft(g_{5}^{2})\be_{3}+\la(\bfs(g_{6}^{2})\otimes \bft(g_{6}^{2}))\al_{2}\al_{3}\al_{4},\\
      \bfs(g_{7}^{3})\otimes \bft(g_{7}^{3})\longmapsto 
         &-\be_{2}(\bfs(g_{4}^{2})\otimes \bft(g_{4}^{2}))
         +(\bfs(g_{7}^{2})\otimes \bft(g_{7}^{2}))\be_{1}.
\end{cases}
\end{align*}

By direct computations, we see that the composite 
$\partial^{n}\partial^{n+1}$ is zero for $n=0,1,2$. 
Therefore,
\begin{equation}
\label{eq:proj-resol-bimod}
\cdots \longrightarrow 
Q^{3}\stackrel{\partial^{3}}{\longrightarrow}
Q^{2}\stackrel{\partial^{2}}{\longrightarrow}
Q^{1}\stackrel{\partial^{1}}{\longrightarrow}
Q^{0}\stackrel{\partial^{0}}{\longrightarrow} 
\La_{{\La^{\rm e}}}
\longrightarrow 0
\end{equation}
is a part of a complex of $\La^{\rm e}$-modules.
Since $(\La/\rad\La) \otimes_\La (Q\up, \partial\up) \iso (P\up, d\up)$ is acyclic,
\eqref{eq:proj-resol-bimod}
is a part of a projective $\La^{\rm e}$-resolution of $\La$ 
by Theorem \ref{thm:ext-GS} (an extended version of \cite[Theorem 2.9]{green2004projective})

%
Moreover, each differential has image contained in the radical of its target projective bimodule. Hence the resulting projective bimodule resolution is minimal.

Applying the functor $ \widehat{(\blank)} :=\Hom_{\La^{\rm e}}(\blank,\La)$ 
to the resolution $(Q\up,\partial\up)$, 
we have the Hochschild complex 
$$
  0 \longrightarrow \widehat{Q^{0}} \stackrel{\widehat{\partial^{1}}}{\longrightarrow} 
  \widehat{Q^{1}} \stackrel{\widehat{\partial^{2}}}{\longrightarrow} 
  \widehat{Q^{2}} \stackrel{\widehat{\partial^{3}}}{\longrightarrow} 
  \widehat{Q^{3}} \longrightarrow \cdots. 
$$
For each $n=0,1,2,3$, 
the $\k$-space $\bigoplus_{g^{n}\in\cG^{n}} \bfs(g^{n})\La \bft(g^{n})$ 
is computed as follows: 
\begin{enumerate}[(i)]
  \item  If $n=0$, then it is given by
\begin{multline*}
         \langle e_{1}, \al_{1}\al_{2}\al_{3}\al_{4} \rangle_{\k} 
         \,\oplus\,
         \langle e_{2}, \al_{2}\al_{3}\al_{4}\al_{1} \rangle_{\k}\\
         \,\oplus\,
         \langle e_{3}, \al_{3}\al_{4}\al_{1}\al_{2} \rangle_{\k}
         \,\oplus\,
         \langle e_{4}, \al_{4}\al_{1}\al_{2}\al_{3} \rangle_{\k}
         \,\oplus\,
         \langle e_{5}, \be_{2}\be_{3}\be_{1} \rangle_{\k},
\end{multline*}
         which has $10$-dimension. 
  \item  If $n=1$, then it is given by
         $$
         \langle \al_{1} \rangle_{\k}
         \,\oplus\,
         \langle \al_{2} \rangle_{\k}
         \,\oplus\,
         \langle \al_{3} \rangle_{\k}
         \,\oplus\,
         \langle \al_{4},\,\be_{3} \rangle_{\k}
         \,\oplus\,
         \langle \be_{1} \rangle_{\k}
         \,\oplus\,
         \langle \be_{2} \rangle_{\k}
         \,\oplus\,
         \langle \al_{4},\, \be_{3} \rangle_{\k},
         $$
         which has $9$-dimension. 
    \item  
    If $n=2$, then it is given by
\begin{multline*}
         \langle e_{1}, \al_{1}\al_{2}\al_{3}\al_{4}(=\be_{1}\be_{2}\be_{3}) \rangle_{\k}
         \,\oplus\,
         \langle \al_{2} \rangle_{\k}
         \,\oplus\,
         \langle \al_{3}\al_{4} \rangle_{\k}
         \,\oplus\,
         \langle \be_{3}\be_{1} \rangle_{\k}\\
         \,\oplus\,
         \langle  e_{4}, \al_{4}\al_{1}\al_{2}\al_{3}(=\la\be_{3}\be_{1}\be_{2}) \rangle_{\k}
         \,\oplus\,
         \langle \al_{4}\al_{1} \rangle_{\k}
         \,\oplus\,
         \langle \be_{2}\be_{3} \rangle_{\k},
\end{multline*}
         which has $9$-dimension. 
    \item
    If $n=3$, then it is given by
\begin{multline*}
         \langle \al_{1}\al_{2} \rangle_{\k}
         \,\oplus\,
         \langle \al_{1}\al_{2}\al_{3}, \be_{1}\be_{2} \rangle_{\k}
         \,\oplus\,
         \langle \al_{2}\al_{3} \rangle_{\k}
         \,\oplus\,
         \langle  \al_{3}\al_{4}\al_{1} \rangle_{\k}\\
         \,\oplus\,
         \langle \al_{4}, \be_{3} \rangle_{\k}
         \,\oplus\,
         \langle \al_{4}, \be_{3} \rangle_{\k}
         \,\oplus\,
         \langle e_{5},\be_{2}\be_{3}\be_{1} \rangle_{\k},
\end{multline*}
         which has $11$-dimension. 
\end{enumerate}
Moreover, for $n=0,1,2,3$, 
we define right $\La^{\rm e}$-module homomorphisms 
by the following equations: 
\begin{enumerate}[(i)]
 \item If $n=0$, then we define $f_{i,j}^{0}$: $Q^{0}\rightarrow \La$
        ($i=1,\ldots,5$, $j=1,2$, $l=1,\ldots,5$) by 
        \begin{itemize}
         \item 
         $
         f_{i,1}^{0}(\bfs(g_{l}^{0})\otimes \bft(g_{l}^{0}))=
          \begin{cases}
           \begin{array}{ll}
            e_{i}& \text{if }l=i,\\
            0    & \text{otherwise}. 
           \end{array}
          \end{cases}
          $
          \item 
          $
          f_{1,2}^{0}(\bfs(g_{l}^{0})\otimes \bft(g_{l}^{0}))=
          \begin{cases}
           \begin{array}{ll}
            \al_{1}\al_{2}\al_{3}\al_{4}(=\be_{1}\be_{2}\be_{3})& \text{if }l=1,\\
            0    & \text{otherwise}. 
           \end{array}
          \end{cases}
          $
          \item 
          $
          f_{2,2}^{0}(\bfs(g_{l}^{0})\otimes \bft(g_{l}^{0}))=
          \begin{cases}
           \begin{array}{ll}
           \al_{2}\al_{3}\al_{4}\al_{1} & \text{if }l=2,\\
            0    & \text{otherwise}. 
           \end{array}
          \end{cases}
          $
          \item 
          $
          f_{3,2}^{0}(\bfs(g_{l}^{0})\otimes \bft(g_{l}^{0}))=
          \begin{cases}
           \begin{array}{ll}
            \al_{3}\al_{4}\al_{1}\al_{2} & \text{if }l=3,\\
            0    & \text{otherwise}. 
           \end{array}
          \end{cases}
          $
          \item 
          $
          f_{4,2}^{0}(\bfs(g_{l}^{0})\otimes \bft(g_{l}^{0}))=
          \begin{cases}
           \begin{array}{ll}
           \al_{4}\al_{1}\al_{2}\al_{3}(=\la \be_{3}\be_{1}\be_{2}) & \text{if }l=4,\\
            0    & \text{otherwise}. 
           \end{array}
          \end{cases}
          $
          \item 
          $
          f_{5,2}^{0}(\bfs(g_{l}^{0})\otimes \bft(g_{l}^{0}))=
          \begin{cases}
           \begin{array}{ll}
            \be_{2}\be_{3}\be_{1} & \text{if }l=5,\\
            0    & \text{otherwise}. 
           \end{array}
          \end{cases}
          $
          \end{itemize}
 \item If $n=1$, then we define $f_{i,1}^{1},\,f_{j,2}^{1}$: $Q^{1}\rightarrow \La$ 
        ($i=1,\ldots,7$, $j=4,7$, $l=1,\ldots,7$) by 
        \begin{itemize}
         \item 
         $
         f_{1,1}^{1}(\bfs(g_{l}^{1})\otimes \bft(g_{l}^{1}))=
          \begin{cases}
           \begin{array}{ll}
            \al_{1}& \text{if }l=1,\\
            0    & \text{otherwise}. 
           \end{array}
          \end{cases}
          $
          \item 
          $
          f_{2,1}^{1}(\bfs(g_{l}^{1})\otimes \bft(g_{l}^{1}))=
          \begin{cases}
           \begin{array}{ll}
            \al_{2}& \text{if }l=2,\\
            0    & \text{otherwise}. 
           \end{array}
          \end{cases}
          $
          \item 
          $
          f_{3,1}^{1}(\bfs(g_{l}^{1})\otimes \bft(g_{l}^{1}))=
          \begin{cases}
           \begin{array}{ll}
            \al_{3} & \text{if }l=3,\\
            0    & \text{otherwise}. 
           \end{array}
          \end{cases}
          $
          \item 
          $
          f_{4,1}^{1}(\bfs(g_{l}^{1})\otimes \bft(g_{l}^{1}))=
          \begin{cases}
           \begin{array}{ll}
            \al_{4} & \text{if }l=4,\\
            0    & \text{otherwise}. 
           \end{array}
          \end{cases}
          $
          \item 
          $
          f_{4,2}^{1}(\bfs(g_{l}^{1})\otimes \bft(g_{l}^{1}))=
          \begin{cases}
           \begin{array}{ll}
           \be_{3} & \text{if }l=4,\\
            0    & \text{otherwise}. 
           \end{array}
          \end{cases}
          $
          \item 
          $
          f_{5,1}^{1}(\bfs(g_{l}^{1})\otimes \bft(g_{l}^{1}))=
          \begin{cases}
           \begin{array}{ll}
            \be_{1} & \text{if }l=5,\\
            0    & \text{otherwise}. 
           \end{array}
          \end{cases}
          $
          \item 
          $
          f_{6,1}^{1}(\bfs(g_{l}^{1})\otimes \bft(g_{l}^{1}))=
          \begin{cases}
           \begin{array}{ll}
            \be_{2} & \text{if }l=6,\\
            0    & \text{otherwise}. 
           \end{array}
          \end{cases}
          $
          \item 
          $
          f_{7,1}^{1}(\bfs(g_{l}^{1})\otimes \bft(g_{l}^{1}))=
          \begin{cases}
           \begin{array}{ll}
            \al_{4} & \text{if }l=7,\\
            0    & \text{otherwise}. 
           \end{array}
          \end{cases}
          $
          \item 
          $
          f_{7,2}^{1}(\bfs(g_{l}^{1})\otimes \bft(g_{l}^{1}))=
          \begin{cases}
           \begin{array}{ll}
            \be_{3} & \text{if }l=7,\\
            0    & \text{otherwise}. 
           \end{array}
          \end{cases}
          $
          \end{itemize}
 \item If $n=2$, then we define $f_{i,1}^{2},\,f_{j,2}^{2}$: 
        $Q^{2}\rightarrow \La$ 
        ($i=1,\ldots,7$, $j=1,5$, $l=1,\ldots,7$) by 
        \begin{itemize}
         \item 
         $
         f_{1,1}^{2}(\bfs(g_{l}^{2})\otimes \bft(g_{l}^{2}))=
          \begin{cases}
           \begin{array}{ll}
            e_{1}& \text{if }l=1,\\
            0    & \text{otherwise}. 
           \end{array}
          \end{cases}
          $
          \item 
          $
          f_{1,2}^{2}(\bfs(g_{l}^{2})\otimes \bft(g_{l}^{2}))=
          \begin{cases}
           \begin{array}{ll}
            \al_{1}\al_{2}\al_{3}\al_{4} & \text{if }l=1,\\
            0    & \text{otherwise}. 
           \end{array}
          \end{cases}
          $
          \item 
          $
          f_{2,1}^{2}(\bfs(g_{l}^{2})\otimes \bft(g_{l}^{2}))=
          \begin{cases}
           \begin{array}{ll}
            \al_{2} & \text{if }l=2,\\
            0    & \text{otherwise}. 
           \end{array}
          \end{cases}
          $
          \item 
          $
          f_{3,1}^{2}(\bfs(g_{l}^{2})\otimes \bft(g_{l}^{2}))=
          \begin{cases}
           \begin{array}{ll}
            \al_{3}\al_{4} & \text{if }l=3,\\
            0    & \text{otherwise}. 
           \end{array}
          \end{cases}
          $
          \item 
          $
          f_{4,1}^{2}(\bfs(g_{l}^{2})\otimes \bft(g_{l}^{2}))=
          \begin{cases}
           \begin{array}{ll}
           \be_{3}\be_{1} & \text{if }l=4, \\
            0    & \text{otherwise}. 
           \end{array}
          \end{cases}
          $
          \item 
          $
          f_{5,1}^{2}(\bfs(g_{l}^{2})\otimes \bft(g_{l}^{2}))=
          \begin{cases}
           \begin{array}{ll}
            e_{4} & \text{if }l=5,\\
            0    & \text{otherwise}. 
           \end{array}
          \end{cases}
          $
          \item 
          $
          f_{5,2}^{2}(\bfs(g_{l}^{2})\otimes \bft(g_{l}^{2}))=
          \begin{cases}
           \begin{array}{ll}
            \al_{4}\al_{1}\al_{2}\al_{3} & \text{if }l=5,\\
            0    & \text{otherwise}. 
           \end{array}
          \end{cases}
          $
          \item 
          $
          f_{6,1}^{2}(\bfs(g_{l}^{2})\otimes \bft(g_{l}^{2}))=
          \begin{cases}
           \begin{array}{ll}
            \al_{4}\al_{1} & \text{if }l=6,\\
            0    & \text{otherwise}. 
           \end{array}
          \end{cases}
          $
          \item 
          $
          f_{7,1}^{2}(\bfs(g_{l}^{2})\otimes \bft(g_{l}^{2}))=
          \begin{cases}
           \begin{array}{ll}
            \be_{2}\be_{3} & \text{if }l=7,\\
            0    & \text{otherwise}. 
           \end{array}
          \end{cases}
          $
          \end{itemize}
 \item If $n=3$, then we define $f_{i,1}^{3},\,f_{j,2}^{3}$: 
       $Q^{3}\rightarrow \La$ 
        ($i=1,\ldots,7$, $j=2,5,6,7$, $l=1,\ldots,7$) by 
        \begin{itemize}
         \item 
         $
         f_{1,1}^{3}(\bfs(g_{l}^{3})\otimes \bft(g_{l}^{3}))=
          \begin{cases}
           \begin{array}{ll}
            \al_{1}\al_{2}& \text{if }l=1,\\
            0    & \text{otherwise}. 
           \end{array}
          \end{cases}
          $
          \item 
          $
          f_{2,1}^{3}(\bfs(g_{l}^{3})\otimes \bft(g_{l}^{3}))=
          \begin{cases}
           \begin{array}{ll}
            \al_{1}\al_{2}\al_{3} & \text{if }l=2,\\
            0    & \text{otherwise}. 
           \end{array}
          \end{cases}
          $
          \item 
          $
          f_{2,2}^{3}(\bfs(g_{l}^{3})\otimes \bft(g_{l}^{3}))=
          \begin{cases}
           \begin{array}{ll}
            \be_{1}\be_{2} & \text{if }l=2,\\
            0    & \text{otherwise}. 
           \end{array}
          \end{cases}
          $
          \item 
          $
          f_{3,1}^{3}(\bfs(g_{l}^{3})\otimes \bft(g_{l}^{3}))=
          \begin{cases}
           \begin{array}{ll}
            \al_{2}\al_{3} & \text{if }l=3,\\
            0    & \text{otherwise}. 
           \end{array}
          \end{cases}
          $
          \item 
          $
          f_{4,1}^{3}(\bfs(g_{l}^{3})\otimes \bft(g_{l}^{3}))=
          \begin{cases}
           \begin{array}{ll}
           \al_{3}\al_{4}\al_{1} & \text{if }l=4,\\
            0    & \text{otherwise}. 
           \end{array}
          \end{cases}
          $
          \item 
          $
          f_{5,1}^{3}(\bfs(g_{l}^{3})\otimes \bft(g_{l}^{3}))=
          \begin{cases}
           \begin{array}{ll}
            \al_{4} & \text{if }l=5,\\
            0    & \text{otherwise}. 
           \end{array}
          \end{cases}
          $
          \item 
          $
          f_{5,2}^{3}(\bfs(g_{l}^{3})\otimes \bft(g_{l}^{3}))=
          \begin{cases}
           \begin{array}{ll}
            \be_{3} & \text{if }l=5,\\
            0    & \text{otherwise}. 
           \end{array}
          \end{cases}
          $
          \item 
          $
          f_{6,1}^{3}(\bfs(g_{l}^{3})\otimes \bft(g_{l}^{3}))=
          \begin{cases}
           \begin{array}{ll}
            \al_{4} & \text{if }l=6,\\
            0    & \text{otherwise}. 
           \end{array}
          \end{cases}
          $
          \item 
          $
          f_{6,2}^{3}(\bfs(g_{l}^{3})\otimes \bft(g_{l}^{3}))=
          \begin{cases}
           \begin{array}{ll}
            \be_{3} & \text{if }l=6,\\
            0    & \text{otherwise}. 
           \end{array}
          \end{cases}
          $
          \item 
          $
          f_{7,1}^{3}(\bfs(g_{l}^{3})\otimes \bft(g_{l}^{3}))=
          \begin{cases}
           \begin{array}{ll}
            e_{5} & \text{if }l=7,\\
            0    & \text{otherwise}. 
           \end{array}
          \end{cases}
          $
          \item 
          $
          f_{7,2}^{3}(\bfs(g_{l}^{3})\otimes \bft(g_{l}^{3}))=
          \begin{cases}
           \begin{array}{ll}
            \be_{2}\be_{3}\be_{1} & \text{if }l=7,\\
            0    & \text{otherwise}. 
           \end{array}
          \end{cases}
          $
          \end{itemize}
\end{enumerate}
\noindent
For each $n=0,1,2,3$, 
we have an isomorphism of $\k$-spaces 
$$
F_{n}: \bigoplus_{g\in \cG^{n}}\bfs(g)\La \bft(g)\stackrel{\sim}{\rightarrow} 
\Hom_{\La^{\rm e}}(Q^{n},\La)
$$ 
given by $F_{n}(\sum_{g\in \cG^{n}}z_{g})(\bfs(g_{i}^{n})\otimes \bft(g_{i}^{n}))=z_{g_{i}^{n}}$, 
where $z_{g}\in \bfs(g)\La \bft(g)$ for $g\in \cG^{n}$. 
Through these isomorphisms, 
we see that 
$$
\dim_{\k}\widehat{Q^{0}}=10,\,
\dim_{\k}\widehat{Q^{1}}={9},\,
\dim_{\k}\widehat{Q^{2}}=9,\,
\dim_{\k}\widehat{Q^{3}}={11}. 
$$
%

By calculating 
$
\widehat{\partial^{1}}(f_{i,j}^{0})=f_{i,j}^{0} \circ \partial^{1}
$, 
we have 
$$
\widehat{\partial^{1}}
(f_{1,1}^{0}\,f_{1,2}^{0}\,f_{2,1}^{0}\,f_{2,2}^{0}\,
f_{3,1}^{0}\,f_{3,2}^{0}\,f_{4,1}^{0}\,f_{4,2}^{0}\,f_{5,1}^{0}\,f_{5,2}^{0})
=
(f_{1,1}^{1}\,f_{2,1}^{1}\,f_{3,1}^{1}\,f_{4,1}^{1}\,f_{4,2}^{1}\,
f_{5,1}^{1}\,f_{6,1}^{1}\,f_{7,1}^{1}\,f_{7,2}^{1})M_{1}, 
$$
where the representation matrix of $M_{1}$ of size $9 \times 10$ is as follows: 
$$
{\small
\left[
\begin{array}{cccccccccc}
{1} & 0 & -1 & 0 & 0 & 0& 0& 0& 0&0\\
0  & 0 & 1  & 0 & -1& 0& 0& 0& 0&0\\
0  & 0 & 0  & 0 &  1& 0&-1& 0& 0&0\\
-1 & 0 &0   & 0 &  0& 0& 1& 0& 0&0\\
0  & 0 &0   & 0 &  0& 0& 0& 0& 0&0\\
1  & 0 &0   & 0 &  0& 0& 0& 0&-1&0\\
0  & 0 &0   & 0 &  0& 0& -1&0& 1&0\\
0  & 0 &0   & 0 &  0& 0& 0& 0& 0&0\\
-1 & 0 &0   & 0 &  0& 0& 1& 0& 0&0
\end{array}
\right]}
\longrightarrow 
{\small
\left[
\begin{array}{cccccccccc}
{1} & 0 & 0 & 0 & 0 & 0& 0& 0& -1&0\\
0  & 0 & 1  & 0 & 0& 0& 0& 0& -1&0\\
0  & 0 & 0  & 0 &  1& 0&0& 0& -1&0\\
0 & 0 &0   & 0 &  0& 0& 1& 0& -1&0\\
0  & 0 &0   & 0 &  0& 0& 0& 0& 0&0\\
0  & 0 &0   & 0 &  0& 0& 0& 0&0&0\\
0  & 0 &0   & 0 &  0& 0& 0&0& 0&0\\
0  & 0 &0   & 0 &  0& 0& 0& 0& 0&0\\
0 & 0 &0   & 0 &  0& 0& 0& 0& 0&0
\end{array}
\right]}. 
$$
We have $\rank M_{1}=\dim_{\k}\Im \widehat{\partial^{1}}=4$. 

Also, by calculating 
$
\widehat{\partial^{2}}(f_{i,j}^{1})=f_{i,j}^{1} \circ \partial^{2}
$, 
we have 
\begin{align*}
&\widehat{\partial^{2}}
(f_{1,1}^{1}\,f_{2,1}^{1}\,f_{3,1}^{1}\,f_{4,1}^{1}\,f_{4,2}^{1}\,f_{5,1}^{1}
\,f_{6,1}^{1}\,f_{7,1}^{1}\,f_{7,2}^{1})\\
&\quad=
(f_{1,1}^{2}\,f_{1,2}^{2}\,f_{2,1}^{2}\,f_{3,1}^{2}\,f_{4,1}^{2}\,f_{5,1}^{2}
\,f_{5,2}^{2}\,f_{6,1}^{2}\,f_{7,1}^{2})M_{2}, 
\end{align*}
where the representation matrix of $M_{2}$ of size $9 \times 9$ is as follows: 
$$
{\small
\left[
\begin{array}{ccccccccc}
 0 & 0 & 0 & 0 & 0 & 0 & 0 & 0 & 0 \\
 1 & 1 & 1 & 1 & 0 & -1 & -1 & 0 & -1 \\
 0 & 0 & 0 & 0 & 0 & 0 & 0 & 0 & 0 \\
 0 & 0 & 0 & 0 & 0 & 0 & 0 & 1 & 0 \\
 0 & 0 & 0 & 0 & 1 & 0 & 0 & 0 & 0 \\
 0 & 0 & 0 & 0 & 0 & 0 & 0 & 0 & 0 \\
 1 & 1 & 1 & 1 & 0 & -1 & -1 & 0 & -1 \\
 0 & 0 & 0 & 0 & 0 & 0 & 0 & 1 & 0 \\
 0 & 0 & 0 & 0 & 1 & 0 & 0 & 0 & 0 \\
\end{array}
\right]
}
\longrightarrow 
{\small
\left[
\begin{array}{ccccccccc}
 1 & 1 & 1 & 1 & 0 & -1 & -1 & 0 & -1 \\
 0 & 0 & 0 & 0 & 1 & 0 & 0 & 0 & 0 \\
 0 & 0 & 0 & 0 & 0 & 0 & 0 & 1 & 0 \\
 0 & 0 & 0 & 0 & 0 & 0 & 0 & 0 & 0 \\
 0 & 0 & 0 & 0 & 0 & 0 & 0 & 0 & 0 \\
 0 & 0 & 0 & 0 & 0 & 0 & 0 & 0 & 0 \\
 0 & 0 & 0 & 0 & 0 & 0 & 0 & 0 & 0 \\
 0 & 0 & 0 & 0 & 0 & 0 & 0 & 0 & 0 \\
 0 & 0 & 0 & 0 & 0 & 0 & 0 & 0 & 0 \\
\end{array}
\right]}. 
$$
We have $\rank M_{2}=\dim_{\k}\Im \widehat{\partial^{2}}={3}$. 

Finally, by calculating 
$
\widehat{\partial^{3}}(f_{i,j}^{2})=f_{i,j}^{2} \circ \partial^{3}
$, 
we have 
\begin{align*}
&\widehat{\partial^{3}}
(f_{1,1}^{2}\,f_{1,2}^{2}\,f_{2,1}^{2}\,f_{3,1}^{2}\,f_{4,1}^{2}\,f_{5,1}^{2}
\,f_{5,2}^{2}\,f_{6,1}^{2}\,f_{7,1}^{2})\\
&\quad=
(f_{1,1}^{3}\,f_{2,1}^{3}\,f_{2,2}^{3}\,f_{3,1}^{3}\,f_{4,1}^{3}\,f_{5,1}^{3}\,{f_{5,2}^{3}}\,f_{6,1}^{3}\,f_{6,2}^{3}\,f_{7,1}^{3}\,f_{7,2}^{3})M_{3}, 
\end{align*}
where the representation matrix of $M_{3}$ of size $11 \times 9$ is as follows: 
$$
{\small
\left[
\begin{array}{ccccccccc}
 1 & 0 & -1 & 0 & 0 & 0 & 0 & 0 & 0 \\
 0 & 0 & 0 & 0 & 0 & 0 & 0 & 0 & 0 \\
 \la & 0 & 0 & 0 & 0 & -1 & 0 & 0 & 0 \\
 0 & 0 & 1 & 0 & 0 & -1 & 0 & 0 & 0 \\
 0 & 0 & 0 & 1 & 0 & 0 & 0 & -1 & 0 \\
 -1 & 0 & 0 & 0 & 0 & 1 & 0 & 0 & 0 \\
 {0} & {0} & {0} & {0} & {0}& {0} & {0} & {0} & {0} \\
 0 & 0 & 0 & 0 & 0 & 0 & 0 & 0 & 0 \\
 -\la & 0 & 0 & 0 & 0 & 1 & 0 & 0 & 0 \\
 0 & 0 & 0 & 0 & 0 & 0 & 0 & 0 & 0 \\
 0 & 0 & 0 & 0 & -1 & 0 & 0 & 0 & 1 
\end{array}
\right]}
\longrightarrow 
{\small
\left[
\begin{array}{ccccccccc}
 1 & 0 & -1 & 0 & 0 & 0 & 0 & 0 & 0 \\
 0 & 0 & 0 & 0 & 0 & 0 & 0 & 0 & 0 \\
 \la-1 & 0 & 0 & 0 & 0 & 0 & 0 & 0 & 0 \\
 -1 & 0 & 1 & 0 & 0 & 0 & 0 & 0 & 0 \\
 0 & 0 & 0 & 1 & 0 & 0 & 0 & -1 & 0 \\
 -1 & 0 & 0 & 0 & 0 & 1 & 0 & 0 & 0 \\
 0 & 0 & 0 & 0 & 0 & 0 & 0 & 0 & 0 \\
 0 & 0 & 0 & 0 & 0 & 0 & 0 & 0 & 0 \\
 -\la+1 & 0 & 0 & 0 & 0 & 0 & 0 & 0 & 0 \\
 0 & 0 & 0 & 0 & 0 & 0 & 0 & 0 & 0 \\
 0 & 0 & 0 & 0 & -1 & 0 & 0 & 0 & 1 \\
\end{array}
\right]}. 
$$
Considering a row reduction of $M_{3}$, 
\begin{align*}
\rank M_{3}&=\dim_{\k}\Im \widehat{\partial^{3}}\\
&=
\begin{cases}
5\quad \text{if $\la\neq 1$},\,\\
4\quad \text{if $\la= 1$}. 
\end{cases}
\end{align*}
By direct calculations, $M_{2}M_{1}=O$ and $M_{3}M_{2}=O$ hold. 

Therefore, for each $n = 0, 1, 2$, we have the dimensional formula for 
the Hochschild cohomology group $\HH^{n}(\La)$: 
\begin{align*}
\dim_{\k}\HH^{0}(\La)
&=\dim_{\k}\Ker \widehat{\partial^{1}}
=\dim_{\k}\widehat{Q^{0}}-\dim_{\k}\Im \widehat{\partial^{1}}\\
&=\dim_{\k}\widehat{Q^{0}}-\rank M_{1}
=10-4=6. 
\end{align*}
\begin{align*}
\dim_{\k}\HH^{1}(\La)
&=\dim_{\k}\Ker \widehat{\partial^{2}}-\dim_{\k}\Im \widehat{\partial^{1}}\\
&=(\dim_{\k}\widehat{Q^{1}}-\dim_{\k}\Im \widehat{\partial^{2}})
  -\dim_{\k}\Im \widehat{\partial^{1}}\\
&=(\dim_{\k}\widehat{Q^{1}}-\rank M_{2})-\rank M_{1}\\
&={9}-{3}-4 = 2.
\end{align*}
\begin{align*}
\dim_{\k}\HH^{2}(\La)
&=\dim_{\k}\Ker \widehat{\partial^{3}}-\dim_{\k}\Im \widehat{\partial^{2}}\\
&=(\dim_{\k}\widehat{Q^{2}}-\dim_{\k}\Im \widehat{\partial^{3}})-\rank M_{2}\\
&=\begin{cases}
9-5-{3}={1}\quad \text{if $\la\neq 1$},\,\\
9-4-{3}={2}\quad \text{if $\la=1$}.
\end{cases}
\end{align*}
\qed

\appendix

\section{Stable Auslander--Reiten quivers of selfinjective algebras}

Here we give a proof of Proposition \ref{prp:stAR-iso-AR(over-socle)'-as-tr-qv} for the benefit of the reader.
Before doing it we prepare some statements.

\begin{lem}
\label{lem:nonprj-ind=mod/soc}
Let $\La$ be a selfinjective algebra, and $M$ an indecomposable right $\La$-module. 
Then $M(\mathrm{soc}\,\La) \ne 0$ if and only if $M$ is projective in $\mod \La$.
\end{lem}
\begin{proof}
($\Rightarrow$).
Let $\{e_{1},\dots, e_n\}$ be an orthogonal set of primitive idempotents of $\La$ with $1_{\La}=e_{1}+\cdots +e_{n}$. 
Then since $\La$ is selfinjective, $e_i\La$ is injective and has a simple socle for all $i = 1, \dots, n$, and
we have $\mathrm{soc}\,\La=\bigoplus_{i=1}^{n} \mathrm{soc}\,(e_{i}\La)$ as right $\La$-modules. 
By the assumption that $M(\soc\La) \ne 0$, 
there exist $i=1,\ldots,n$ and $m \in M$ such that 
$m \cdot \soc e_i\La \ne 0$. 
Then we can define a homomorphism $f \colon e_{i}\La \to M$ of right $\La$-modules
by $e_{i}a\mapsto me_{i}a$ for all $a \in \La$. 
By construction, $f(\soc e_{i}\La)= m\soc e_{i}\La \ne 0$.
Thus $\soc e_i \La \not\subseteq \Ker f$.
Since $\soc e_i \La$ is a simple module,
$(\soc e_i \La) \cap \Ker f = 0$.
Since $\soc e_i \La$ is an essential submodule of $e_i \La$,
we have $\Ker f = 0$.
Therefore, $f$ is a monomorphism.
Since $ e_i\La$ is injective, it is a direct summand of $M$,
and hence we have $M \iso e_i\La$ because $M$ is indecomposable.

\noindent
($\Leftarrow$) 
If $M$ is projective in $\mod \La$, then
there exists a primitive idempotent $e$ 
such that $M \iso e\La$.
Then we have $M(\soc\La) \iso (e\La)(\soc\La) = e(\soc\La) = \soc e\La \ne 0$. 
\end{proof}

\begin{ntn}
For a finite-dimensional algebra $A$, we denote by
$\mod A$ the category of finitely generated right $A$-modules, and
by $\ind A$ a complete set of isomorphism classes of representatives of indecomposable right $A$-modules.
\end{ntn}

\begin{prp}
\label{prp:nonproj-ind=mod/soc}
Let $\La$ be a selfinjective algebra.
Then we have a disjoint union
$$
\ind \La = \ind (\La/\soc\La) \sqcup \calP_{\La},
$$
where $\calP_\La$ is a complete set of isomorphism classes of indecomposable projective right modules.
\end{prp}
\begin{proof}
The canonical surjective homomorphism $\La \to \La/\soc\La$ yields
the fully faithful embedding $\mod (\La/\soc\La) \to \mod \La$ whose image coincides with $\{M \in \mod \La \mid M(\soc\La) = 0\}$.
This shows the assertion by Lemma \ref{lem:nonprj-ind=mod/soc}. 
\end{proof}
\begin{ntn}
For a quiver $\sfQ$, we denote by $\sfQ_0$ (resp.\ $\sfQ_1$)
the set of vertices (resp.\ arrows) of $\sfQ$.

For a finite-dimensional algebra $A$, we denote by
$\Ga_A$ (resp.\ $\stAR{A}$) the AR-quiver (resp.\ stable AR-quiver) of $A$. 
\end{ntn}

By Proposition \ref{prp:nonproj-ind=mod/soc}, we may regard as follows: 
\begin{equation}
\label{eq:stAR=over-soc}
(\stAR{\La})_{0}
=(\ind\La)\setminus \calP_\La
=\ind(\La/\soc\La)
=(\Ga_{\La/\soc\La})_{0}.
\end{equation}
\begin{lem}
\label{lem:ass-same}
Let $\La$ be a selfinjective algebra, and
$$
(\ast)\quad 0 \to X \ya{f} Y \ya{g} Z \to 0
$$
an exact sequence in $\mod(\La/\soc\La)$
$($or equivalently, an exact sequence in $\mod \La$ such that any $W \in \{X, Y, Z\}$ has no nonzero projective direct summands$)$.
Then the following are equivalent:
\begin{enumerate}[{\rm(1)}]
\item 
$(*)$ is an almost split sequence in $\mod(\La/\soc\La)$.
\item
$(\ast)$ is an almost split sequence in $\mod\La$.
\end{enumerate}
Thus, there exists a translation $X \dashleftarrow Z$
in $\Ga_{\La/\soc \La}$ if and only if there exists one
in $\stAR{\La}$ arising from an almost split sequence
of the form $(*)$ in $\mod \La$ whose middle term has
no nonzero projective direct summands.
%
\end{lem}
\begin{proof}
(1) $\Rightarrow$ (2).
It is enough to prove that $g$ is right almost split in $\mathrm{mod}\,\La$. 
Let $h \colon W \to Z$ be a non-retraction with $W \in \ind\La$. 
If $W(\soc \La) = 0$, then 
$W \in \mod (\La/\soc\La)$, and $g \colon Y \to Z$ turns out to be a morphism in $\mod \La/\soc\La$.
Then there exists a morphism 
$h' \colon W \to Y$ in $\mod (\La/\soc\La)$
such that $h=g h'$ by assumption.
If $W(\soc \La) \ne 0$, then, by Lemma \ref{lem:nonprj-ind=mod/soc}, 
there exists a primitive idempotent $e\in \La$ such that $W \cong e \La$. 
Since $e\La$ is projective in $\mod\La$, 
there exists a morphism 
$h'\colon W \to Y$ 
such that $h = g h'$ in $\mod \La$. 
Therefore, in any case, $g$ is right almost split.

(2) $\Rightarrow$ (1).
This follows from Proposition \ref{prp:nonproj-ind=mod/soc} and the fact 
that the inclusion functor $\mod(\La/\soc\La) \to \mod\La$ is fully faithful and exact. 
\end{proof}

For an algebra $A$, we denote by $\prj A$ (resp.\ $\inj A$) the category of finitely generated projective (injective) right $A$-modules.

\begin{lem}
\label{lem:prj-inj-La-/soc}
Let $\La$ be a connected selfinjective algebra, which is not semisimple.
Then the functor $\prj \La \to \prj (\La/\soc\La)$ sending
$P$ to $P/P(\soc\La)$ induces a bijection 
$$
(\prj \La)/\iso \ \to\  (\prj (\La/\soc\La))/\iso.
$$
Dually, the functor $\inj \La \to \inj(\La/\soc\La)$ sending
$I$ to $\rad I$ induces a bijection 
$$
(\inj \La)/\iso \ \to\  (\inj (\La/\soc\La))/\iso.
$$
\end{lem}

\begin{proof}
Let $\{e_i,\dots, e_n\}$ be a basic set of primitive idempotents in $\La$.
Then
$$
|(\prj \La)/\iso| = n = |(\inj \La)/\iso|.
$$
Since $\La$ is not semisimple, we have $\soc \La \le \rad \La$.
Hence we also have
$$
|(\prj (\La/\soc\La))/\iso| = n = | (\inj (\La/\soc\La))/\iso|.
$$
Note that
$$
\begin{aligned}
(\prj \La)/\iso &= \{[e_i\La]\mid i=1,\dots,n\}\\
(\prj (\La/\soc\La))/\iso &=\{[\ovl{e_i}(\La/\soc\La)]\mid i=1,\dots,n\}\\
(\inj \La)/\iso &= \{[D(\La e_i)] \mid i=1,\dots, n\},\quad \text{and}\\
(\inj(\La/\soc\La))/\iso &= \{[D((\La/\soc\La)\ovl{e_i})] \mid i=1,\dots, n\},
\end{aligned}
$$
where $\ovl{e_i}:= e_i + \soc\La \in \La/\soc\La$, and the bracket denotes the isoclass.
Then the isomorphisms
$e_i\La/(e_i\La)(\soc \La) \iso \ovl{e_i}(\La/\soc\La)$ and
$\rad D(e_i\La) \iso D((\La/\soc\La)\ovl{e_i})$ for all $i=1,\dots, n$ show the assertion.
\end{proof}

\begin{prp}
\label{prp:stAR-iso-AR(over-socle)-as-val-qv}
Let $\La$ be a selfinjective algebra. 
Then we have 
$$
\stAR{\La} \iso \Ga_{\La/\soc\La}
$$
as valued quivers $($not necessarily as translation quivers$)$.
\end{prp}

\begin{proof}
We construct an isomorphism
$$
F \colon \stAR{\La} \to \Ga_{\La/\soc\La}
$$
of valued quivers.
The equality \eqref{eq:stAR=over-soc} gives an identification between
$(\stAR{\La})_0$ and $(\Ga_{\La/\soc\La})_0$.
Using this we define $F_0 \colon (\stAR{\La})_0 \to (\Ga_{\La/\soc\La})_0$
to be the identity.
To define a map $F_1 \colon (\stAR{\La})_1 \to (\Ga_{\La/\soc\La})_1$,
we show that the following conditions are equivalent for any $a \ge 1$:
\begin{enumerate}
\item 
There exists an arrow $X \ya{(a,b)} Y$ in $\Ga_{\La/\soc\La}$ with a value
$(a,b)$ for some $b \ge 1$.
\item 
There exists a source map 
$$
X \to Y^{(a)} \ds Y' 
$$
such that any direct summand of $Y'$ 
is not isomorphic to $Y$. 
\item 
When $X$ is non-injective in $\mod (\La/\soc \La)$,
there exists an almost split sequence in $\mod(\La/\soc \La)$ of the form
$$
0 \to X \to Y^{(a)} \ds Y' \to Z \to 0;
$$
or when 
$X$ is injective in $\mod \La/\soc \La$ ,
there exists an isomorphism $f \colon X \to \rad e\La$ for some primitive idempotent $e$ by Lemma \ref{lem:prj-inj-La-/soc}, and $X/\soc X \iso Y^{(a)} \ds Y'$,
where in each case, any direct summand of $Y'$ 
is not isomorphic to $Y$. 
\item 
When $X$ is not isomorphic to $\rad e\La$ for any primitive idempotent $e \in \La$,
there exists an almost split sequence in $\mod \La$ of the form
$$
0 \to X \to Y^{(a)} \ds Y' \to Z \to 0
$$
in $\mod \La$ such that any $W \in \{X, Y, Y', Z\}$ has no nonzero projective direct summands;

or when there exists an isomorphism $f \colon X \to \rad e\La$ for some primitive idempotent $e$, and $X/\soc X \iso Y^{(a)} \ds Y'$,
where in each case, any direct summand of $Y'$ 
is not isomorphic to $Y$. 
\item 
There exists an arrow $X\ya{(a,b)} Y$ in $\stAR{\La}$ with a value
$(a,b)$ for some $b \ge 1$.
\end{enumerate}

The equivalences between (1), (2) and (3) are trivial.

(3) $\Leftrightarrow$ (4).
Note that $X$ is injective in $\mod(\La/\soc \La)$ if and only if
there exists an isomorphism $f \colon X \to \rad e\La$ for some primitive idempotent $e$.
Then in the first case, this equivalence follows from Lemma \ref{lem:ass-same} 
and in the second case, it is trivial.

The equivalence of (4) and (5) follows by definition.

Dually, the same argument for a sink map shows that
there exists an arrow $X \to Y$ in $\stAR{\La}$ with
a value $(d, b)$ for some $d \ge 1$.

Altogether, we see that there exists an arrow $X \to Y$ with
a value $(a, b)$ in $\Ga_{\La/\soc\La}$ if and only if so does in
$\stAR{\La}$.

Hence we can define a map $F_1 \colon (\stAR{\La})_1 \to (\Ga_{\La/\soc\La})_1$ as the identity map.
It is obvious that $F = (F_0, F_1) \colon \stAR{\La} \to \Ga_{\La/\soc\La}$ turns out to be a quiver isomorphism, and it preserves the values on arrows as shown above.
\end{proof}

Proposition \ref{prp:stAR-iso-AR(over-socle)-as-val-qv} is improved as follows.

\begin{prp}[Proposition \ref{prp:stAR-iso-AR(over-socle)'-as-tr-qv}]
\label{prp:stAR-again}
Let $\La$ be a selfinjective algebra, which is not semisimple. 
Then we have 
$$
\stAR{\La} \iso \Ga_{\La/\soc\La}'
$$
as valued translation quivers, where $\Ga_{\La/\soc\La}'$ is constructed
from $\Ga_{\La/\soc\La}$ by adding the translations
$$
D((\La/\soc \La)e) \dashleftarrow e(\La/\soc \La)
$$
for all primitive idempotents $e$ in $\La/\soc \La$.
\end{prp}

\begin{proof}
By Proposition \ref{prp:stAR-iso-AR(over-socle)-as-val-qv},
it suffices to show that there exists a translation
$X \dashleftarrow Z$
in $\stAR{\La}$ if and only if so does in $\Ga'_{\La/\soc\La}$.
By Lemma \ref{lem:ass-same}, we easily see the following statements are equivalent:
\begin{enumerate}
\item 
There exists a translation
$X \dashleftarrow Z$
in $\stAR{\La}$.

\item 
There exists a translation
$X \dashleftarrow Z$
in $\Ga_{\La}$.

\item 
There exists an almost split sequence of the form
$0\to X \to Y \to Z \to 0$ in $\mod \La$.

\item 
One of the following holds
\begin{enumerate}
\item 
There exists an almost split sequence of the form
$0\to X \to Y \to Z \to 0$ in $\mod(\La/\soc\La)$.
\item
$X \iso \rad P$ and $Z \iso P/\soc P$ in $\mod \La$ for some projective indecomposable $\La$-module $P$.
In this case, we have an almost split sequence of the form
$$
0 \to \rad P \to P \ds (\rad P/\soc P) \to P/\soc P \to 0
$$
in $\mod \La$.
\end{enumerate}
\item 
There exists a translation
$X \dashleftarrow Z$
in $\Ga'_{\La/\soc\La}$.
\end{enumerate}
\end{proof}

\section{An extension of Green--Snashall's Theorem}

We will extend Proposition 2.8 and Theorem 2.9 in \cite{green2004projective}
by Green--Snashall in this appendix, which is used in
this paper.
Throughout this section, $\La$ is a finite-dimensional algebra over a field $\k$
with the Jacobson radical $\r$.
We denote by $\mod \La^{\op}$ the category of finitely generated
left $\La$-modules.

\begin{thm}[cf.\ Proposition 2.8 and Theorem 2.9 in \cite{green2004projective}]
\label{thm:ext-GS}
Consider a complex of $\La$-$\La$-bimodules
$$
Q\down:= (\cdots\to  Q_3 \ya{A_3} Q_2 \ya{A_2}Q_1 \ya{A_1} Q_0 \ya{A_0} M\ya{A_{-1}} 0 \to \cdots)
$$
such that $A_0$ is an epimorphism, $Q_i$ are finitely generated projective $\La$-$\La$-bimodules for all $i\ge 0$, and $M$ is projective in $\mod \La^{\op}$.
If $\La/\r \ox_{\La}Q\down$ is
acyclic,
then so is $Q\down$, i.e., $Q\down$ is a projective bimodule resolution of $M$.
\end{thm}

\begin{proof}
We set $T:= \La/\r \ox_\La \blank$ for short.
Assume that $T(Q\down)$ is acyclic.
By induction on $n \ge 0$ we show that $\Im A_n = \Ker A_{n-1}$
and $\Ker A_{n-1}$ is projective in  $\mod \La^{\op}$, which completes
the proof.

\subsection*{Case $n = 0$}
Since $A_0$ is an epimorphism we have $\Im A_0 = M = \Ker A_{-1}$, and $\Ker A_{-1} = M$ is projective in $\mod\La^{\op}$.
\subsection*{Case $n \ge 1$}
We have an exact sequence
\begin{equation}\label{eq:ex-seq}
0 \to \Ker A_{n-1}\overset{\ta_{n-1}}{\hookrightarrow} Q_{n-1} \ya{A_{n-1}} \Im A_{n-1}\to 0.
\end{equation}
By the induction hypothesis $\Im A_{n-1} = \Ker A_{n-2}$
is projective in $\mod \La^{\op}$, which shows that the sequence
\eqref{eq:ex-seq} splits in $\mod \La^{\op}$.
Hence $\Ker A_{n-1}$ turns out to be projective in $\mod \La^{\op}$
as desired because so is $Q_{n-1}$.
Moreover the sequence \eqref{eq:ex-seq} gives us a split exact sequence
$$
0 \to T(\Ker A_{n-1})\ya{T(\ta_{n-1})} T(Q_{n-1}) \ya{T(A_{n-1})} T(\Im A_{n-1})\to 0
$$
in $\mod \La$.
In particular $T(\ta_{n-1})$ is a monomorphism.
Using this sequence we show
\begin{equation}\label{eq:HT}
H_{n-1}(T( Q\down)) \iso T(H_{n-1}(Q\down)).
\end{equation}
Apply $T$ to the commutative diagram
$$
\vcenter{
\xymatrix{
Q_n && Q_{n-1}&& Q_{n-2}\\
&\Im A_n & \Ker A_{n-1}& \Im A_{n-1} &\Ker A_{n-2}
\ar"1,1";"1,3"^{A_n}
\ar@{->>}"1,1";"2,2"_{\ro_n}
\ar@{^{(}->}"2,2";"2,3"_{\si_n}
\ar@{^{(}->}"2,3";"1,3"_{\ta_{n-1}}
\ar"1,3";"1,5"^{A_{n-1}}
\ar@{->>}"1,3";"2,4"^{\ro_{n-1}}
\ar@{=}"2,4";"2,5"_{\si_{n-1}}
\ar@{^{(}->}"2,5";"1,5"_{\ta_{n-2}}
}}
$$
to have the commutative diagram
$$
\vcenter{
\xymatrix{
T(Q_n) && T(Q_{n-1})&& T(Q_{n-2})\\
&T(\Im A_n) & T(\Ker A_{n-1})& T(\Im A_{n-1}) &T(\Ker A_{n-2}).
\ar"1,1";"1,3"^{T(A_n)}
\ar@{->>}"1,1";"2,2"_{T(\ro_n)}
\ar"2,2";"2,3"_{T(\si_n)}
\ar@{>->}"2,3";"1,3"_{T(\ta_{n-1})}
\ar"1,3";"1,5"^{T(A_{n-1})}
\ar@{->>}"1,3";"2,4"^{T(\ro_{n-1})}
\ar@{=}"2,4";"2,5"_{T(\si_{n-1})}
\ar@{>->}"2,5";"1,5"_{T(\ta_{n-2})}
}}
$$
Looking at this diagram we see
$$
\begin{aligned}
Z_{n-1}(T(Q\down)) &=\Ker T(A_{n-1}) = \Ker T(\ro_{n-1})=T(\ta_{n-1})(T(\Ker A_{n-1})),\ \text{and}\\
B_{n-1}(T(Q\down)) &= T(\ta_{n-1})(\Im T(\si_n)).
\end{aligned}
$$
Therefore we have the following diagram with exact rows
$$
\xymatrix{
T(\Im A_n) &&T(\Ker A_{n-1}) & T(H_{n-1}(Q\down)) &0\\
&\Im T(\si_n)\\
&B_{n-1}(T(Q\down)) & Z_{n-1}(T(Q\down)) & H_{n-1}(T(Q\down)) & 0,
\ar"1,1";"1,3"^{T(\si_n)}
\ar"1,3";"1,4"^{\text{can}}
\ar"1,4";"1,5"
\ar@{->>}"1,1";"2,2"
\ar@{^{(}->}"2,2";"1,3"
\ar"2,2";"3,2"_{T(\ta_{n-1})}^{\wr}
\ar"1,3";"3,3"^{T(\ta_{n-1})}_{\wr}
\ar@{-->}"1,4";"3,4"
\ar@{^{(}->}"3,2";"3,3"
\ar"3,3";"3,4"^{\text{can}}
\ar"3,4";"3,5"
}
$$
where the vertical maps are isomorphisms because
$T(\ta_{n-1})$ is a monomorphism.
Hence by the universality of the cokernel and by five lemma
$T(\ta_{n-1})$ induces an isomorphism \eqref{eq:HT}.
Since $T(Q\down)$
is acyclic,
the isomorphism \eqref{eq:HT} implies that
$0 = T(H_{n-1}(Q\down)) \iso H_{n-1}(Q\down)/\r H_{n-1}(Q\down)$.
Here note that $H_{n-1}(Q\down)$ is a finitely generated left $\La$-module
by assumption.
Then by Nakayama's lemma we have $H_{n-1}(Q\down) = 0$, i.e.,
$\Im A_{n} = \Ker A_{n-1}$.
\end{proof}



\bigskip

\address{Department of Mathematics, Faculty of Science, Shizuoka University,
836 Ohya, Suruga-ku, Shizuoka, 422-8529, Japan;}

\address{
Institute for Advanced Study, KUIAS, Kyoto University,
Yoshida Ushinomiya-cho, Sakyo-ku, Kyoto 606-8501, 
Japan; and}

\address{Osaka Central Advanced Mathematical Institute,
3-3-138 Sugimoto, Sumiyoshi-ku,
Osaka, 558-8585, Japan.}

\email{asashiba.hideto@shizuoka.ac.jp}

\bigskip

\address{
Katsushika Division, 
Institute of Arts and Sciences, 
Tokyo University of Science, 
6-3-1 Niijuku, Katsushika-ku, Tokyo, 125-8585, Japan.}

\address{
Department of Mathematics, Graduate School of Science, 
Tokyo University of Science, 1-3 Kagurazaka,
Shinjuku-ku, Tokyo, 162-8601, Japan.}

\email{itaba@rs.tus.ac.jp}

\bigskip

\address{
Departamento de Matemáticas,
Universidad de Murcia, Aptdo. 4021, 30100 Espinardo, Murcia, Spain}

\email{msaorinc@um.es}

\end{document}